\documentclass[12pt]{amsart}
\usepackage{amsfonts}
\usepackage{amsmath,amssymb,amsthm,latexsym,graphicx,graphics}
\usepackage{multicol}
\usepackage{color}
\usepackage{url,hyperref}
\usepackage{euscript}
\usepackage{wrapfig}
\usepackage{caption}
\usepackage{url,hyperref}
\usepackage{euscript,color}

\newtheorem{theorem}{Theorem}[section]

\newtheorem{corollary}[theorem]{Corollary}

\newtheorem{definition}[theorem]{Definition}

\newtheorem{proposition}[theorem]{Proposition}
\newtheorem{remark}[theorem]{Remark}

\begin{document}

\title{The sub-Riemannian geometry of screw motions with constant pitch}\thanks{This work was supported by Consejo Nacional de Investigaciones
	Cient\'ificas y T\'ecnicas and Secretar\'{\i}a de Ciencia y T\'ecnica de la
	Universidad Nacional de C\'ordoba.}

\author{Eduardo Hulett}
\author{Ruth Paola Moas}
\author{Marcos Salvai}

\date{}
\begin{abstract}We consider a family of Riemannian manifolds $M$ such that for each unit speed geodesic $\gamma $ of $M$ there exists a distinguished bijective
	correspondence $L$ between infinitesimal translations along $\gamma $ and
	infinitesimal rotations around it. The simplest examples are $\mathbb{R}^{3}$%
	, $S^{3}$ and hyperbolic $3$-space, with $L$ defined in terms of the cross
	product. More generally, $M$ is a connected compact semisimple Lie group, or its non-compact dual,
	or Euclidean space acted on transitively by some group which is contained
	properly in the full group of rigid motions. 
	
	Let $G$ be the identity
	component of the isometry group of $M$. A curve in $G$ may be thought of as
	a motion of a body in $M$. Given $\lambda \in \mathbb{R}$, we define a left
	invariant distribution on $G$ accounting for infinitesimal roto-translations
	of $M$ of pitch $\lambda $. We give conditions for the controllability of
	the associated control system on $G$ and find explicitly all the geodesics
	of the natural sub-Riemannian structure. We also study a similar system on $%
	\mathbb{R}^{7}\rtimes SO\left( 7\right) $ involving the octonionic cross
	product. In an
	appendix we give a friendly presentation of the non-compact dual of a
	compact classical group, as a set of \textquotedblleft small
	rotations\textquotedblright.
 \end{abstract}
%
%
\maketitle

\section{Introduction and statement of the results}

\subsection{Motivating examples}

Let $M_{\kappa }$ be the three dimensional space form of sectional curvature 
$\kappa =0,1,-1$, that is, $M_{0}=\mathbb{R}^{3}$, $M_{1}=S^{3}$ and $M_{-1}$
is hyperbolic space $H^{3}$. Let $G_{\kappa }$ be the identity component of
the isometry group of $M_{\kappa }$. For $\kappa =\pm 1$, if $M_{\kappa }$
is presented as usual as the connected component of $e_{0}$ of $\left\{ x\in 
\mathbb{R}^{4}\mid \left\langle x,x\right\rangle _{\kappa }=\kappa \right\} $%
, where 
\begin{equation*}
	\left\langle x,y\right\rangle _{\kappa }=\kappa
	x_{0}y_{0}+x_{1}y_{1}+x_{2}y_{2}+x_{3}y_{3}\text{,}
\end{equation*}%
then $G_{1}=SO\left( 4\right) $ and $G_{-1}=O_{o}\left( 1,3\right) $.
Identifying $\mathbb{R}^{3}$ with $\left\{ x\in \mathbb{R}^{4}\mid
x_{0}=1\right\} $, we have 
\begin{equation*}
	G_{0}=\left\{ \left( 
	\begin{array}{cc}
		1 & 0 \\ 
		a & A%
	\end{array}%
	\right) \mid a\in \mathbb{R}^{3}\text{, }A\in SO\left( 3\right) \right\} 
	\text{.}
\end{equation*}

Let $\mathfrak{g}_{\kappa }$ be the Lie algebra of $G_{\kappa }$ and let $%
\mathfrak{g}_{\kappa }=\mathfrak{p}_{\kappa }\oplus \mathfrak{k}_{\kappa }$
be the Cartan decomposition associated with $e_{0}$, into infinitesimal
translations of $M_{\kappa}$ through $e_{0}$ and infinitesimal rotations
around that point. Denoting $\mathfrak{o}\left( n\right) =\left\{ X\in 
\mathbb{R}^{n\times n}\mid X^{t}=-X\right\} $, for any $\kappa =0,1,-1$ we
have%
\begin{equation*}
	\mathfrak{p}_{\kappa }=\left\{ \left( 
	\begin{array}{cc}
		0 & -\kappa x^{t} \\ 
		x & 0%
	\end{array}%
	\right) \mid x\in \mathbb{R}^{3}\right\} \text{\ \ \ and\ \ \ \ }\mathfrak{k}%
	_{\kappa }=\left\{ \left( 
	\begin{array}{cc}
		0 & 0 \\ 
		0 & X%
	\end{array}%
	\right) \mid X\in \mathfrak{o}\left( 3\right) \right\} \text{.}
\end{equation*}%
For $x\in \mathbb{R}^{3}$ we set%
\begin{equation*}
	L_{x}\in \mathfrak{0}\left( 3\right) \text{,\ \ \ \ \ \ \ \ }L_{x}\left(
	y\right) =x\times y\text{, \ \ for }y\in \mathbb{R}^{3}\text{.}
\end{equation*}

\begin{definition}
	For $\lambda \in \mathbb{R}$, the $\lambda $-screw distribution on $%
	G_{\kappa }$ is the left invariant distribution $\mathcal{D}^{\lambda }$
	given at the identity by%
	\begin{equation*}
		\mathcal{D}_{e}^{\lambda }=\left\{ D_{\lambda }\left( x\right) :=\left( 
		\begin{array}{cc}
			0 & -\kappa x^{t} \\ 
			x & \lambda L_{x}%
		\end{array}%
		\right) \mid x\in \mathbb{R}^{3}\right\} \text{.}
	\end{equation*}
\end{definition}

Given $x\in \mathbb{R}^{3}$, let $\sigma _{x}$ be the geodesic in $M_{\kappa
}$ through $e_{0}$ with initial velocity $x\in \mathbb{R}^{3}\equiv
T_{e_{0}}M_{\kappa }$. Then, $t\mapsto \exp \left( tD_{\lambda }\left(
x\right) \right) $ is a monoparametric subgroup of $G_{\kappa }$ consisting
of roto-translations along and around $\sigma _{x}$, with pitch equal to $%
\lambda $.

\begin{remark}
	Thinking of the elements of $G_{\kappa }$ as positions of a body in $%
	M_{\kappa }$ with reference state at $e_{0}$, then, according to the
	constraints induced by $\mathcal{D}^{\lambda }$, at the infinitesimal level,
	the body can be translated $c$ units of length along any direction only if
	at the same time it rotates around that direction through an angle $\lambda
	c $.
\end{remark}

\smallskip

Let $\mathcal{D}$ be a smooth distribution on a manifold $N$. A smooth curve 
$\sigma $ in $N$ is said to be \emph{admissible} or \emph{horizontal} if $\sigma ^{\prime }\left( t\right) \in \mathcal{D}_{\sigma \left(
	t\right) }$ for all $t$. One says that the control system in $N$ determined
by $\mathcal{D}$ is \emph{controllable} if for each pair of points in $N$
there exists a piecewise admissible curve joining them.

\begin{proposition}
	\label{ControlR3}The control system $\left( G_{\kappa },\mathcal{D}^{\lambda
	}\right) $ is controllable if and only if $\kappa ^{2}\neq \lambda $.
\end{proposition}

Cf.\ the controllability condition for infinitesimally helicoidal motions
with fixed pitch of oriented geodesics of a space form \cite{AS}.

\medskip

For the following definitions and facts we refer to \cite{RM,
	AgrachevLibroNuevo}; see also \cite{Alek}. Let $\mathcal{D}$ be a smooth
distribution on a manifold $N$. The Chow-Rashevsky Theorem gives a
sufficient condition for the system $\left( N,\mathcal{D}\right) $ to be
controllable: that $\mathcal{D}$ is bracket generating, that is, the vector
fields in $\mathcal{D}$ generate the Lie algebra of vector fields of $N$.

Let $\mathcal{D}$ be a bracket generating distribution on a manifold $N$. A 
\emph{sub-Riemannian} structure on $\left( N,\mathcal{D}\right) $ is an
assignment $g$ of a positive definite inner product $g_{p}$ on each subspace 
$\mathcal{D}_{p}$, varying smoothly with $p\in N$. The length of a
horizontal curve $\gamma :\left[ a,b\right] \rightarrow N$ is defined by 
$	\text{length}\left( \gamma \right) =\int_{a}^{b}\sqrt{g\left( \gamma
		^{\prime }\left( t\right) ,g^{\prime }\left( t\right) \right) }~dt\text{.}$

For arbitrary $p,q\in N$, the expression 
\begin{equation*}
d\left( p,q\right) =\inf \left\{ \text{length}\left( \gamma \right) \mid
	\gamma \text{ is a horizontal piecewise smooth curve joining }p\text{ with }q\right\}
\end{equation*}%
defines a distance on $N$. A constant speed horizontal curve $\gamma $ in $N$
is said to be a \emph{geodesic} if it minimizes the length locally, that is,
if $a<b$ in the domain of $\gamma $ are close enough, then the length of $%
\left. \gamma \right\vert _{\left[ a,b\right] }$ equals the distance between 
$\gamma \left( a\right) $ and $\gamma \left( b\right) $. The following
proposition describes the maximal geodesics of $\left( G_{\kappa },\mathcal{D%
}^{\lambda }\right) $ endowed with the natural sub-Riemannian structure.

\begin{proposition}
	Suppose that $\kappa ^{2}\neq \lambda $ and that the sub-Riemannian
	structure on $\left( G_{\kappa },\mathcal{D}^{\lambda }\right) $ is the left
	invariant metric whose norm at the identity is given by $\left\Vert
	D_{\lambda }\left( x\right) \right\Vert =\left\Vert x\right\Vert $ for all $%
	x\in \mathbb{R}^{3}$. Then a curve in $G_{\kappa }$ is a sub-Riemannian
	geodesic through the identity if and only if it equals $\gamma _{x,y}$ for
	some $x,y\in \mathbb{R}^{3}$, where, for all $t$,%
	\begin{equation*}
		\gamma _{x,y}\left( t\right) =\exp \left( t\left( 
		\begin{array}{cc}
			0 & -\kappa x^{t} \\ 
			x & L_{\lambda x+y}%
		\end{array}%
		\right) \right) \exp \left( t\left( 
		\begin{array}{cc}
			0 & 0 \\ 
			0 & -L_{y}%
		\end{array}%
		\right) \right) \text{.}
	\end{equation*}
\end{proposition}

\subsection{Manifolds with distinguished screw motions\label{1.2}}

We consider a broad family of manifolds generalizing the examples above.

Let $M$ be a complete Riemannian manifold and let $\gamma :\mathbb{R}%
\rightarrow M$ be a unit speed geodesic of $M$. An isometry $\varphi $ of $M$
is said to be a \emph{translation along} $\gamma $ if there exists $%
t_{o}\in \mathbb{R}$ such that $\varphi \left( \gamma \left( t\right)
\right) =\gamma \left( t+t_{o}\right) $ and $\left( d\varphi \right)
_{\gamma \left( t\right) }$ realizes the parallel transport along $\gamma $
from $T_{\gamma \left( t\right) }M$ to $T_{\gamma \left( t+t_{o}\right) }M$,
for all $t\in \mathbb{R}$. The abundance of translations characterizes the
symmetric spaces: The manifold $M$ is symmetric if and only if for any
geodesic $\gamma $ in $M$ there is a monoparametric group of isometries of $%
M $ consisting of translations along $\gamma $.

We consider symmetric spaces $M$ such that for each unit speed geodesic $%
\gamma $ in $M$ there is a distinguished bijection between infinitesimal
translations along $\gamma $ and infinitesimal rotations around it and we
study the sub-Riemannian geometry associated with special screw motions of $%
M $.

From now on, $K$ will be a connected compact semisimple Lie group and $%
\mathfrak{k}$ its Lie algebra.

Let $K^{\mathbb{C}}$ be the complexification of $K$. The maximal compact Lie
group of $K^{\mathbb{C}}$ is $K$ and its Lie algebra is $i\mathfrak{k+k}$,
the complexification of $\mathfrak{k}$. For instance, for the classical
groups $K=SO\left( n\right) $, $SU\left( n\right) $ and $Sp\left( n\right) $%
, we have $K^{\mathbb{C}}=SO\left( n,\mathbb{C}\right) $, $SL\left( n,%
\mathbb{C}\right) $ and $Sp\left( 2n,\mathbb{C}\right) $, respectively.

Let $\mathfrak{k}\rtimes _{\text{Ad}}K$ be the the Cartan motion group of $K$%
, that is, $\mathfrak{k}\times K$ endowed with the operation $\left(
w,B\right) \cdot \left( z,A\right) =\left( w+\text{Ad}\left( B\right)
z,BA\right) $. Its Lie algebra is $\mathfrak{k}\rtimes _{\text{ad}}\mathfrak{%
	k}$, with Lie bracket given by $\left[ \left( x,y\right) ,\left( u,v\right) %
\right] =\left( \left[ x,v\right] +\left[ y,u\right] ,\left[ y,v\right]
\right) $.

The connected compact semisimple Lie group $K$ gives rise to three symmetric
pairs: 
\begin{equation}
	\left( K\times K,\Delta _{+}\left( K\right) \right) \text{,\ \ \ \ \ \ \ \ \ 
	}\left( K^{\mathbb{C}},K\right) \text{\ \ \ \ \ \ \ \ \ and \ \ \ \ \ \ \ \ }%
	\left( \mathfrak{k}\rtimes _{\text{Ad}}K,K\right)  \label{GkK}
\end{equation}%
(given a group $N$, we denote as usual $\Delta _{\pm }\left( N\right)
=\left\{ \left( x,x^{\pm 1}\right) \mid x\in N\right\} $). We call them $%
\left( G_{K,1},H_{K,1}\right) $, $\left( G_{K,-1},H_{K,-1}\right) $ and $%
\left( G_{K,0},H_{K,0}\right) $, respectively, according to the sign $k\in
\left\{ 1,-1,0\right\} $ of the scalar curvature of the symmetric space $%
G_{K,k}/H_{K,k}$, which we call $M_{K,k}$.

We have that $M_{K,1}$ may be identified with $K$ (through the action $%
\left( k_{1},k_{2}\right) \cdot k=k_{1}kk_{2}^{-1}$) and $M_{K,0}$ with $%
\mathfrak{k}$, i.e., it is Euclidean space where we do not consider the full
group of rigid motions, but only $\mathfrak{k}\rtimes _{\text{Ad}}K$. The
manifold $K^{\mathbb{C}}/K$ is the non-compact dual symmetric space of $K$.
In the appendix we give a more amiable presentation of this quotient when $K$
is a classical compact Lie group.

Henceforth, we sometimes omit the subindices $K$ or $k$, if there is no
danger of confusion, when dealing with properties that are shared.

Let $\mathfrak{g}$ be the Lie algebra of $G$ and let $\mathfrak{g}=\mathfrak{%
	p}\oplus \mathfrak{h}$ be the Cartan decomposition associated with the point 
$H\in M$, into infinitesimal translations through $H$ and infinitesimal rotations around $H$. We describe $\mathfrak{p}$ and $\mathfrak{h}$ in the following
table, together with a distinguished linear map $L:\mathfrak{p}\rightarrow 
\mathfrak{h}$ satisfying $\left[ L\left( X\right) ,X\right] =0$ for all $%
X\in \mathfrak{p}$.%
\begin{equation}
	\begin{tabular}{|l|l|l|l|l|l|l|}
		\hline
		$M$ & $G$ & $H$ & $\mathfrak{g}$ & $\mathfrak{p}$ & $\mathfrak{h}$ & $L:%
		\mathfrak{p}\rightarrow \mathfrak{h}$ \\ \hline
		$K$ & $K\times K$ & $\Delta _{+}\left( K\right) $ & $\mathfrak{k}\times 
		\mathfrak{k}$ & $\Delta _{-}\left( \mathfrak{k}\right) $ & $\Delta
		_{+}\left( \mathfrak{k}\right) $ & $L\left( x,-x\right) =\left( x,x\right) $
		\\ \hline
		$K^{\mathbb{C}}/K$ & $K^{\mathbb{C}}$ & $K$ & $\mathfrak{k+ik}$ & $i%
		\mathfrak{k}$ & $\mathfrak{k}$ & $L\left( ix\right) =x$ \\ \hline
		$\mathfrak{k}$ & $\mathfrak{k}\rtimes _{\text{Ad}}K$ & $\left\{ 0\right\}
		\times K$ & $\mathfrak{k}\rtimes _{\text{ad}}\mathfrak{k}$ & $\mathfrak{k}%
		\times \left\{ 0\right\} $ & $\left\{ 0\right\} \times \mathfrak{k}$ & $%
		L\left( x,0\right) =\left( 0,x\right) $ \\ \hline
	\end{tabular}
	\label{table}
\end{equation}

In Proposition \ref{identif} below we present a condensed version of the Lie
algebras of $G_{K,k}$, together with the operators $L$, that will be useful
in the proofs. Nevertheless, we have chosen to introduce them in this more
concrete and natural way. For the sake of brevity, we usually write $L_{X}$
instead of $L\left( X\right) $.

\begin{theorem}
	\label{main}Let $K$ be a connected compact semisimple Lie group and let $%
	\lambda \in \mathbb{R}$. For each symmetric pair $\left( G,H\right) $ in the
	table above and the corresponding operator $L$, let $\mathcal{D}^{\lambda }$
	be the left invariant distribution on $G$ given at the identity by%
	\begin{equation*}
		\mathcal{D}_{e}^{\lambda }=\left\{ X+\lambda L_{X}\mid X\in \mathfrak{p}%
		\right\} \subset \mathfrak{g}\text{.}
	\end{equation*}
	
	Then the control system on $G$ determined by $\mathcal{D}^{\lambda }$ is
	controllable, except for the pair $\left( K\times K,\Delta _{+}\left(
	K\right) \right) $ with $\lambda =\pm 1$, and the pair $\left( \mathfrak{k}%
	\rtimes _{\text{\emph{Ad}}}K,\left\{ 0\right\} \times K\right) $ with $%
	\lambda =0$.
\end{theorem}

In order to define on $\left( G,\mathcal{D}^{\lambda }\right) $ a
sub-Riemannian structure, we consider on $\mathfrak{k}$ the bi-invariant
canonical inner product. This is the opposite of the Killing form on $%
\mathfrak{k}$ (for the classical cases it is a multiple of $\left\langle
x,y\right\rangle =-$ Re$~$tr$\left( x^{\ast }y\right) $, $x^{\ast }$ being
the conjugate transpose of $x$). We pass it to $\mathfrak{p}$ via the
canonical morphism between $\mathfrak{k}$ and $\mathfrak{p}$, namely, 
\begin{equation*}
	\mathfrak{k}\rightarrow \Delta _{-}\left( \mathfrak{k}\right) \text{,\ }%
	x\mapsto \left( x,-x\right) \text{,\ \ \ \ }\mathfrak{k}\rightarrow i%
	\mathfrak{k}\text{, }x\mapsto ix\text{\ \ \ \ and\ \ \ \ }\mathfrak{k}%
	\rightarrow \mathfrak{k}\times \left\{ 0\right\} \text{, }x\mapsto \left(
	x,0\right) \text{.}
\end{equation*}

The sub-Riemannian structure on $G$ is defined as the left invariant
quadratic form on the distribution $\mathcal{D}^\lambda$ given at the identity by $%
\left\Vert \left( X,\lambda L_{X}\right) \right\Vert =\left\Vert
X\right\Vert $,\ for $X\in \mathfrak{p}$.

We have determined all sub-Riemannian geodesics of $\left( G_{K,k},\mathcal{D%
}^{\lambda }\right) $ explicitly:

\begin{theorem}
	\label{Geo1}Suppose that $\lambda ^{2}\neq k$. Let $G_{k}$ be endowed with
	the left invariant sub-Riemannian structure defined at $e$ by the
	distribution $\mathcal{D}^{\lambda }$ with 
	\begin{equation*}
		\mathcal{D}_{e}^{\lambda }=\left\{ X+\lambda L_{X}\mid X\in \mathfrak{p}%
		\right\} \text{\ \ \ \ and\ \ \ \ \ }\left\Vert X+\lambda L_{X}\right\Vert
		=\left\Vert X\right\Vert
	\end{equation*}
	for $X\in \mathfrak{p}$. Then a curve $\gamma $ in $G_{k}$ with $\gamma
	\left( 0\right) =e$ is a sub-Riemannian geodesic through the identity if and
	only if%
	\begin{equation*}
		\gamma \left( t\right) =\exp \left( t\left( X+\lambda L_{X}+L_{Y}\right)
		\right) \exp \left( -tL_{Y}\right)
	\end{equation*}%
	for some $X,Y\in \mathfrak{p}$. Moreover, $\gamma \left( t\right) =\exp
	\left( t\left( X+\lambda L_{X}\right) \right) $ for all $t$ if and only if $%
	\left[ X,Y\right] =0$.
\end{theorem}

\begin{remark}
	\label{Remark1}Expressions of sub-Riemannian geodesics as products of
	exponentials can be found in several books and papers, for example, in \emph{%
		\cite{JurBook, Brockett, RM, Alibro, BoscainCG, Boscain, AM, GMG, Domokos, JMF,
			HMSl, Sachkov22}}. In many of them the case $\lambda =0$ \emph{(}pure
	translations\emph{)} is considered and usually further geometric properties
	are studied.
	The novelty of our result lies in the fact that we resort to an ad-hoc non-standard bi-invariant
	pseudo-Riemannian metric on $G$; for instance, it is not a multiple of the
	Killing form of $G$ when this is nondegenerate \emph{(}$k\neq 0$\emph{)}. Besides, 
	in our situation there is no inclusion relationship between $%
	\left[ \mathcal{D}^{\lambda },\mathcal{D}^{\lambda }\right] $ and $\left( 
	\mathcal{D}^{\lambda }\right) ^{\perp }$, except for $\lambda =0$.
\end{remark}

\subsection{Octonionic screw motions of $\mathbb{R}^{7}$}

Screw motions appear in another interesting context apart from the ones
studied above, induced by the octonionic cross product on $\mathbb{R}^{7}=%
\operatorname{Im}\left( \mathbb{O}\right) $, where $\mathbb{O}$ denotes the skew
field of the octonions, the biggest among the normed division algebras. In
Section \ref{SOcto} we recall the definition and some properties of the
octonionic cross product $\times :\mathbb{R}^{7}\times \mathbb{R}%
^{7}\rightarrow \mathbb{R}^{7}$. For $u\in \mathbb{R}^{7}$ we define $L_{u}:%
\mathbb{R}^{7}\rightarrow \mathbb{R}^{7}$ by $L_{u}\left( v\right) =u\times
v $. We have that $L_{u}\in \mathfrak{o}\left( 7\right) $. Next we look at a
situation similar to the one in Subsection \ref{1.2}, but now $L:\mathbb{R}%
^{7}\rightarrow \mathfrak{o}\left( 7\right) $ is no longer surjective (we
define a 7-dimensional distribution on Lie group of dimension 28).

\begin{theorem}
	\label{Octo}Let $G=\mathbb{R}^{7}\rtimes SO_{7}$ and $\lambda \neq 0$. Then
	the control system on $G$ determined by the left invariant distribution $%
	\mathcal{D}^{\lambda }$ on $G$ given at the identity $\left( 0,I_{7}\right) $
	by%
	\begin{equation*}
		\mathcal{D}_{\left( 0,I_{7}\right) }^{\lambda }=\left\{ \left( x,\lambda
		L_{x}\right) \mid x\in \operatorname{Im}\left( \mathbb{O}\right) \right\}
	\end{equation*}%
	is controllable.
\end{theorem}

We endow $\left( G,\mathcal{D}^{\lambda }\right) $ with a left invariant
sub-Riemannian structure analogous to the the one we considered above: $%
\left\Vert \left( \left( x,\lambda L_{x}\right) \right) \right\Vert
=\left\Vert x\right\Vert $. In order to state our result on sub-Riemannian
geodesics we introduce the Lie group $G_{2}$, the automorphism group of the
octonionic cross product, which is 
\begin{equation*}
	G_{2}=\left\{ A\in GL\left( 7,\mathbb{R}\right) \mid A\left( x\right) \times
	A\left( y\right) =A\left( x\times y\right) \text{ for all }x,y\in \mathbb{R}%
	^{7}\right\} \text{.}
\end{equation*}%
See in \cite{G,BG} the role of this group in the classification of sub-Riemannian model spaces.
Let $\mathfrak{g}_{2}$ be the Lie algebra of $G_{2}$, which is contained in $\mathfrak{o%
}\left( 7\right) $. We will see in (\ref{decom}) that the operators $L_{x}$
form a subspace complementary to $\mathfrak{g}_{2}$. That is the reason why
we do not consider the group $\mathbb{R}^{7}\rtimes G_{2}$ acting on $%
\mathbb{R}^{7}$ for our purposes.

Let $Z:\mathbb{R}^{7}\times \mathbb{R}^{7}\rightarrow \mathfrak{o}\left(
7\right) $ be defined by%
\begin{equation}
	Z\left( u,v\right) =3u\wedge v-L_{u\times v}\text{,}  \label{Zxy}
\end{equation}%
where $u\wedge v\in \mathfrak{o}\left( 7\right) $ is given by $\left(
u\wedge v\right) \left( w\right) =\left\langle w,u\right\rangle
v-\left\langle w,v\right\rangle u$ (if $\left\{ u,v\right\} $ is
orthonormal, then $u\wedge v$ rotates through a right angle the plane
spanned by the vectors and vanishes on the orthogonal complement). We will
see in (\ref{decompo}) that $Z\left( u,v\right) \in \mathfrak{g}_{2}$ for
all $u,v\in \mathbb{R}^{7}$.

We have found explicitly some sub-Riemannian geodesics.

\begin{theorem}
	\label{OctoGeode}For any $x,y\in \mathbb{R}^{7}$ with $x\perp y$, the curve%
	\begin{equation*}
		\gamma _{x,y}\left( t\right) =\exp \left( t\left( x,\lambda L_{x}+Z\left(
		x,y\right) \right) \right) \exp \left( t\left( 0,-Z\left( x,y\right) \right)
		\right) 
	\end{equation*}%
	is a sub-Riemannian geodesic of $\left( \mathbb{R}^{7}\rtimes SO\left(
	7\right) ,\mathcal{D}^{\lambda }\right) $ with initial velocity $\left(
	x,\lambda L_{x}\right) $.
\end{theorem}

\begin{remark} 		This theorem provides a nontrivial application of \cite{Toth} 
	and \cite{AP}. They give conditions 
	for the homogeneity of sub-Riemannian geodesics without resorting to the approach used in
	the concrete examples in the books and papers cited in Remark \emph{\ref{Remark1}}. They do not invoke
	a suitable taming \emph{(}pseudo\emph{)}-Riemannnian metric on the manifold, whose geodesics 
	are known explicitly \emph{(}see \cite{GMG}\emph{)}. For the screw octonionic system  we have not been able to find such a metric on $\mathbb{R}^{7}\rtimes SO\left( 7\right) $; in particular,  we have
	not succeeded in using the usual tools to find a bi-invariant metric or prove its nonexistence 
	\cite{O}.  
	
	We do not know whether the remaining geodesics through the
	identity are also orbits of monoparametric subgroups of isometries of the
	sub-Riemannian structure; for general left invariant metrics, those
	geodesics are rather the exception.
\end{remark}

In contrast with the Riemannian case, sub-Riemannian geodesics are not
determined by their initial velocities, but by their initial momenta in $\mathfrak{g}^{\ast }$. 
In Proposition \ref{momentumO} below
we see to which of them the
geodesics in the theorem are associated.

\smallskip

The third author thanks Jorge Lauret, who many years ago made him aware of 
the fact that dimension 3 is not necessary for having a nice correspondence $%
L$ as in this paper.

\section{Controllability of $\left( G_{K,k},\mathcal{D}^{\protect\lambda %
	}\right) $ and its sub-Riemannian geodesics\label{S1}}

\subsection{A unified presentation of the Lie algebras of $G_{K,k}$}

Let $K$ be, as before, a connected compact semisimple Lie group with Lie
algebra $\mathfrak{k}$. Let $\mathfrak{g}_{\mathfrak{k},k}$ be the Lie
algebra of $G_{K,k}$ as in (\ref{GkK}), that is, 
\begin{equation*}
	\mathfrak{g}_{\mathfrak{k,}1}=\mathfrak{k\times k}\text{,\ \ \ \ \ \ \ \ }%
	\mathfrak{g}_{\mathfrak{k},-1}=\mathfrak{k+ik}\text{\ \ \ \ \ \ \ \ and\ \ \
		\ \ \ \ \ }\mathfrak{g}_{\mathfrak{k,}0}=\mathfrak{k\rtimes }_{\text{ad}}%
	\mathfrak{k}\text{.}
\end{equation*}
We fix $K$ and $\mathfrak{k}$, so we can omit referring to them in the
notation if there is not danger of confusion.

For each $k=1,-1,0$, we present in Proposition \ref{identif} below an
isomorphism of $\mathfrak{g}_{\mathfrak{k},k}$ with a Lie algebra $\mathfrak{%
	k}_{k}$, which will allow us to handle the three cases simultaneously more
easily. Let $\mathfrak{k}_{k}$ be the direct sum $\mathfrak{k}\oplus 
\mathfrak{k}$ endowed with the bracket%
\begin{equation}
	\left[ \left( x,y\right) ,\left( u,v\right) \right] _{k}=\left( \left[ x,v%
	\right] +\left[ y,u\right] ,\left[ y,v\right] +k\left[ x,u\right] \right) 
	\text{.}  \label{corchete_k}
\end{equation}%
We comment that it is isomorphic to $\left\{ \left( 
\begin{array}{cc}
	y & kx \\ 
	x & y%
\end{array}%
\right) \mid x,y\in \mathfrak{k}\right\} $ with the commutator as the
bracket.

\begin{proposition}
	\label{identif}For each $k=1,-1,0$, $\mathfrak{k}_{k}$ is a Lie algebra and
	there exists a Lie algebra isomorphism $T_{k}:\mathfrak{k}_{k}\rightarrow 
	\mathfrak{g}_{k}$ such that $T_{k}\left( \mathfrak{k\times }\left\{
	0\right\} \right) =\mathfrak{p}$, $T_{k}\left( \left\{ 0\right\} \times 
	\mathfrak{k}\right) =\mathfrak{k}$ and $T_{k}\left( \mathcal{E}^{\lambda
	}\right) =\mathcal{D}^{\lambda }$, where $\mathcal{E}^{\lambda }=\left\{
	\left( x,\lambda x\right) \mid x\in \mathfrak{k}\right\} $.
\end{proposition}

\begin{proof}
	We already know that the Lie algebra of $\mathfrak{k}\rtimes _{\text{Ad}}K$
	is $\mathfrak{k}_{0}$. It is not difficult to show that $T_{1}:\mathfrak{k}%
	_{k}\rightarrow \mathfrak{k}\oplus \mathfrak{k}$, and $T_{-1}:\mathfrak{k}%
	_{k}\rightarrow \mathfrak{k}^{\mathbb{C}}$ given by%
	\begin{equation*}
		T_{1}\left( x,y\right) =\tfrac{1}{2}\left( x-y,x+y\right) \text{,\ \ \ \ \ \
			\ }T_{-1}\left( x,y\right) =y+ix
	\end{equation*}%
	satisfy the required conditions.
\end{proof}


We have introduced the family $\mathfrak{g}_{\mathfrak{k,}k}$ of Lie
algebras, together with the operators $L$, spread out over the three rows of
Table \ref{table}, appealing to familiar geometric objects. Next we
characterize succinctly their elements.

\begin{proposition}
	Let $\left( \mathfrak{g,k}\right) $ be a symmetric pair of Lie algebras and
	let $\mathfrak{g}=\mathfrak{p}\oplus \mathfrak{k}$ be the associated Cartan
	decomposition. Let $k=1,-1,0$ and suppose that there exists a linear
	isomorphism $L:\mathfrak{p}\rightarrow \mathfrak{k}$ such that%
	\begin{equation}
		k\left[ L\left( x\right) ,L\left( y\right) \right] =\left[ x,y\right] \text{%
			\ \ \ \ \ \ and\ \ \ \ \ \ }\left[ L\left( x\right) ,z\right] =L\left[ x,z%
		\right]  \label{displayIso}
	\end{equation}%
	for all $x,y\in \mathfrak{p}$, $z\in \mathfrak{k}$. Then $\phi :\mathfrak{g}%
	\rightarrow \mathfrak{k}_{k}$, $\phi \left( x+z\right) =\left( L\left(
	x\right) ,z\right) $, is a Lie algebra isomorphism. Conversely, the
	operators $L$ in the last column of the table in \emph{(\ref{table})}
	satisfy the identities \emph{(\ref{displayIso})}.
\end{proposition}

\begin{proof}
	Recall that $\left[ \mathfrak{p},\mathfrak{p}\right] \subset \mathfrak{k}$, $%
	\left[ \mathfrak{p},\mathfrak{k}\right] \subset \mathfrak{p}$, since $%
	\mathfrak{g}=\mathfrak{p}\oplus \mathfrak{k}$ is a Cartan decomposition for
	a symmetric pair $\left( \mathfrak{g,k}\right) $. By the hypothesis we have%
	\begin{equation*}
		L\left( \left[ x,w\right] +\left[ z,y\right] \right) =\left[ L\left(
		x\right) ,w\right] +L\left[ z,y\right] =\left[ L\left( x\right) ,w\right] -%
		\left[ L\left( y\right) ,z\right] \text{.}
	\end{equation*}%
	Hence $\phi \left[ x+z,y+w\right] =\left[ \phi \left( x+z\right) ,\phi
	\left( y+w\right) \right] _{k}$. Straightforward computations yield the
	assertion about the operators $L$.
\end{proof}


\subsection{Controllability of the system $\left( G_{K,k},\mathcal{D}^{%
		\protect\lambda }\right) $}

Next we present the proof of the theorem giving conditions for the
controllability of the system $\left( G_{K,k},\mathcal{D}^{\lambda }\right) $%
.

\medskip

\noindent \textit{Proof of Theorem \ref{main}. }By the Chow-Rashevsky
Theorem, it suffices to show that the distribution $\mathcal{D}^{\lambda }$
is bracket generating in the stated cases. By Proposition \ref{identif}, it
is enough to verify that the subspace $\mathcal{E}^{\lambda }=\left\{ \left(
x,\lambda x\right) \mid x\in \mathfrak{k}\right\} $ of $\mathfrak{k}_{k}$
satisfies $\mathcal{E}^{\lambda }+\left[ \mathcal{E}^{\lambda },\mathcal{E}%
^{\lambda }\right] =\mathfrak{k}_{k}$. For $x,y\in \mathfrak{k}$, according
to (\ref{corchete_k}), we compute%
\begin{equation}
	\left[ \left( x,\lambda x\right) ,\left( y,\lambda y\right) \right] =\left(
	2\lambda \left[ x,y\right] ,\left( \lambda ^{2}+k\right) \left[ x,y\right]
	\right) \text{.}  \label{corcheteD}
\end{equation}

\noindent Now, $\left[ \mathfrak{k,k}\right] =\mathfrak{k}$, since $%
\mathfrak{k}$ is semisimple. Hence, the subspace $\left\{ \left( 2\lambda
z,\left( \lambda ^{2}+k\right) z\right) \mid z\in \mathfrak{k}\right\} $ is
contained in $\left[ \mathcal{E}^{\lambda },\mathcal{E}^{\lambda }\right] .$
A dimension counting argument implies that we need only to see that its
intersection with $\mathcal{E}^{\lambda }$ is trivial. We have that $\left(
u,\lambda u\right) =\left( 2\lambda z,\left( \lambda ^{2}+k\right) z\right) $
implies that $\left( \lambda ^{2}-k\right) z=0$. Consequently, $\mathcal{E}%
^{\lambda }$ is bracket generating unless either $k=1$ and $\kappa =\pm 1$,
or $k=0$ and $\lambda =0$, as desired. \hfill $\square $

\medskip

\noindent \textit{Proof of Proposition \ref{ControlR3}. }It is a corollary
of Theorem \ref{main} for $\mathfrak{k}=\mathfrak{o}\left( 3\right) $, since
in this case $\mathfrak{k}\oplus \mathfrak{k}$, $\mathfrak{k}^{\mathbb{C}}$
and $\mathfrak{k\rtimes }_{\text{ad}}\mathfrak{k}$ are isomorphic as Lie
algebras to $\mathfrak{o}\left( 4\right) ,$ $\mathfrak{o}\left( 1,3\right) $
and $\mathbb{R}^{3}\rtimes \mathfrak{o}\left( 3\right) $, respectively.
\hfill $\square $

\subsection{Sub-Riemannian geodesics of $\left( G_{K,k},\mathcal{D}^{\protect%
		\lambda }\right) $}

We begin by introducing a bi-invariant metric on $G_{K,k}$ which will be
useful to prove the theorem giving explicitly the sub-Riemannian geodesics
of $\left( G_{K,k},\mathcal{D}^{\lambda }\right) $.

\begin{proposition}
	For $\lambda \in \mathbb{R}$ and $k=1,-1,0$, the inner product $g_{\lambda
		,k}$ on $\mathfrak{k}_{k}$ defined by%
	\begin{equation}
		g_{\lambda ,k}\left( \left( x,y\right) ,\left( u,v\right) \right) =\lambda
		\left\langle x,v\right\rangle +\lambda \left\langle y,u\right\rangle
		-\left\langle y,v\right\rangle -k\left\langle x,u\right\rangle  \label{gBi}
	\end{equation}%
	is bi-invariant and it is degenerate if and only if $\lambda ^{2}=k$.
\end{proposition}

\begin{proof}
	Since the inner product on $\mathfrak{k}$ is bi-invariant by hypothesis, $%
	\left\langle \left[ x,u\right] ,v\right\rangle +\left\langle u,\left[ x,v%
	\right] \right\rangle $ holds for all $x,u,v\in \mathfrak{k}$. A
	straightforward computation using (\ref{corchete_k}) yields then that $%
	g_{\lambda ,k}\left( \left[ \left( x,y\right) ,\left( u,v\right) ,\right]
	_{k},\left( u,v\right) \right) =0$ for all $x,y,u,v\in \mathfrak{k}$. This
	implies the first assertion.
	
	If $\lambda ^{2}=k$ on sees that $g_{\lambda ,k}\left( \left( x,y\right)
	,\left( u,\lambda u\right) \right) =0$ for all $x,y,u$, and so the inner
	product is degenerate. Conversely, suppose that there exists $\left(
	u,v\right) \neq 0$ which is orthogonal to any $\left( x,y\right) \in 
	\mathfrak{k}_{k}$. Plugging in (\ref{gBi}) arbitrary non-zero elements $%
	\left( x,0\right) $ and $\left( 0,y\right) $ we have%
	\begin{equation*}
		0=\lambda \left\langle x,v\right\rangle -k\left\langle x,u\right\rangle
		=\left\langle x,\lambda v-ku\right\rangle \text{ \ \ \ and \ \ \ }0=\lambda
		\left\langle y,u\right\rangle -\left\langle y,v\right\rangle =\left\langle
		y,\lambda u-v\right\rangle \text{,}
	\end{equation*}%
	respectively. Hence, $ku=\lambda v=\lambda ^{2}u$, since the inner product
	on $\mathfrak{k}$ is nondegenerate. Now, $u\neq 0$ (otherwise, $\left(
	u,v\right) =0$) and so $\lambda ^{2}=k$, as desired.
\end{proof}

\begin{remark}
	\label{NabN}We have presented sub-Riemannian geodesics in the natural manner
	through the local length minimization property. When it comes to study them,
	one is compelled to consider two types, normal and abnormal geodesics \emph{(%
	}see for instance \emph{\cite{RM})}. We will use two results that give \emph{%
		normal} geodesics, namely, Proposition \emph{\ref{PropRM}} and Corollary \emph{\ref%
	{CoroToth}}. However, in our situations,  Theorem \emph{\ref{Geo1}} and
	Theorem \emph{\ref{OctoGeode}}, we do not need to care about abnormal geodesics,
	since in both cases the distributions have step 2 \emph{(}by the proofs of
	the controllability of the corresponding systems\emph{)} and this implies
	that if there exist abnormal geodesic, they must be normal as well \emph{(}%
	see for example \emph{20.5.1} of \emph{\cite{Alibro})}.
\end{remark}

The explicit form of the sub-Riemannian geodesics will follow from
Proposition 11.19 in \cite{RM} (see other references with similar statements
in Remark \ref{Remark1}):

\begin{proposition}
	\label{PropRM}\emph{\cite{RM} }Let $G$ be Lie group endowed with a
	bi-invariant pseudo-Rieman\-nian metric $g$. Let $K$ be a closed subgroup of 
	$G$ and let $\mathfrak{k}$ and $\mathfrak{g}$ be their respective Lie
	algebras. Suppose that $\mathfrak{d}=\mathfrak{k}^{\perp }$ is bracket
	generating and the metric on $\mathfrak{d}$ is positive definite. Then all
	the normal sub-Riemannian geodesics of the sub-Riemannian manifold $\left( G,%
	\mathfrak{d},\left. g\right\vert _{\mathfrak{d}\times \mathfrak{d}}\right) $
	through the identity have the form%
	\begin{equation*}
		t\mapsto \exp \left( t\left( u+z\right) \right) \exp \left( -tz\right)\text{,%
		}
	\end{equation*}%
	with $u\in \mathfrak{d}$ and $z\in \mathfrak{k}$.
\end{proposition}

\smallskip

\noindent \textit{Proof of Theorem \ref{Geo1}.} By Proposition \ref{identif}
we may consider $\mathfrak{k}_{k}$ and $\mathcal{E}^{\lambda }$ instead of $%
\mathfrak{g}_{k}$ and $\mathcal{D}^{\lambda }$. Let $h_{\lambda ,k}=\frac{1}{%
	\lambda ^{2}-k}~g_{\lambda ,k}$, which is a bi-invariant pseudo-Riemannian
metric on $\mathfrak{k}_{k}\equiv \mathfrak{g}_{k}$. We compute%
\begin{eqnarray*}
	\left( \lambda ^{2}-k\right) h_{\lambda ,k}\left( \left( x,\lambda x\right)
	,\left( 0,v\right) \right) &=&\lambda \left\langle x,v\right\rangle
	-\left\langle \lambda x,v\right\rangle =0\text{,} \\
	\left( \lambda ^{2}-k\right) h_{\lambda ,k}\left( \left( x,\lambda x\right)
	,\left( x,\lambda x\right) \right) &=&\lambda \left\langle x,\lambda
	x\right\rangle -k\left\langle x,x\right\rangle =\left( \lambda ^{2}-k\right)
	\left\Vert x\right\Vert ^{2}\text{.}
\end{eqnarray*}%
Hence, $h_{\lambda ,k}\left( \mathcal{E}^{\lambda },\left\{ 0\right\} \times 
\mathfrak{k}\right) =0$ and $h_{\lambda ,k}\left( \left( x,\lambda x\right)
,\left( x,\lambda x\right) \right) =\left\Vert x\right\Vert ^{2}$ for all $%
x\in \mathfrak{k}$. Then, by Proposition \ref{PropRM} with $\frak{d}=\mathcal{E}^{\lambda }$, all the normal
sub-Riemannian geodesics through the identity have the stated form. Those
are all the geodesics through the identity (see Remark \ref{NabN}). The last
assertion follows from the fact that $e^{t\left( A+B\right) }=e^{tA}e^{tB}$
for all $t$ if and only if $\left[ A,B\right] =0$ (see for instance p.\ 23
in \cite{EngelBook}). \hfill $\square $

\section{Octonionic screw motions\label{SOcto}}

\subsection{The octonionic cross product and its automorphism group}

We recall the octonionic cross product on $\mathbb{R}^{7}=\operatorname{Im}\left( 
\mathbb{O}\right) $. With respect to the ordered basis $\left\{ e_{1},\dots
,e_{7}\right\} $, it is given by%
\begin{equation}
	e_{i}\times e_{i+1}=e_{i+3}\text{ \ mod 7,}  \label{multiplicaO}
\end{equation}%
$e_{i}\times e_{j}=-e_{j}\times e_{i}$ for all $i,j$ (in particular, $%
e_{i}\times e_{i}=0$) and $e_{i}\times e_{j}=e_{k}$ implies $e_{j}\times
e_{k}=e_{i}$. See for instance \cite{Eschen}. For $u\in \mathbb{R}^{7} $ we
define $L_{u}:\mathbb{R}^{7} \rightarrow \mathbb{R}^{7}$ by $L_{u}\left(
v\right) =u\times v$. We have that $L_{u}\in \mathfrak{o}\left( 7\right) $.

Before proving Theorem \ref{Octo} we recall some facts concerning the
octonionic cross product and $\mathfrak{g}_{2}$, the Lie algebra of $G_{2}$.
By (2.16) in \cite{SK},%
\begin{equation}
	\mathfrak{g}_{2}=\left\{ Z\in \mathfrak{o}\left( 7\right) \mid Z\left(
	u\times v\right) =Z\left( u\right) \times v+u\times Z\left( v\right) \text{
		for all }u,v\in \mathbb{R}^{7}\right\} \text{.}  \label{Lieg2}
\end{equation}%
This implies that%
\begin{equation}
	\left[ Z,L_{u}\right] =L_{Z\left( u\right) }  \label{corZL}
\end{equation}%
for all $Z\in \mathfrak{g}_{2}$ and $u\in \mathbb{R}^{7}$. Moreover, by
Theorem 8.5 of \cite{S}, $\mathfrak{o}\left( 7\right) $ decomposes as 
\begin{equation}
	\mathfrak{o}\left( 7\right) =\mathcal{L}\oplus \mathfrak{g}_{2}\text{,}
	\label{decom}
\end{equation}%
where $\mathcal{L}=\left\{ L_{y}\mid y\in \mathbb{R}^{7}\right\} $. A
straightforward computation using (4.2) in \cite{S} gives%
\begin{equation*}
	\left[ L_{u},L_{v}\right] =3u\wedge v-2L_{u\times v}\text{.}
\end{equation*}%
Now we rewrite this expression as%
\begin{equation}
	\left[ L_{u},L_{v}\right] +L_{u\times v}=3u\wedge v-L_{u\times v}=Z\left(
	u,v\right) \text{,}  \label{decompo}
\end{equation}%
with $Z\left( u,v\right) $ as in (\ref{Zxy}). By (5.4) in \cite{CDF}, the
left hand side is in $\mathfrak{g}_{2},$ and so $Z\left( u,v\right) \in 
\mathfrak{g}_{2}$, as stated in the introduction. We conclude then from (\ref%
{decom}) that $-L_{u\times v}=L_{-u\times v}$ and $Z\left( u,v\right) $ are
the components of $\left[ L_{u},L_{v}\right] $ in $\mathcal{L}$ and $%
\mathfrak{g}_{2}$, respectively.

\subsection{Controllability of the octonionic system}

We prove the controllability of the system $\left( \mathbb{R}^{7}\rtimes
SO_{7},\mathcal{D}^{\lambda }\right) $.

\medskip

\noindent \textit{Proof of Theorem \ref{Octo}. }By the Chow-Rashevsky
Theorem it suffices to prove that $\mathcal{D}^{\lambda }$ is
bracket-generating. Similarly as in (\ref{corcheteD}), we have 
\begin{equation*}
	\left[\left( x,\lambda L_{x}\right) ,\left( y,\lambda L_{y}\right) \right]
	=\left( 2\lambda x\times y,\lambda ^{2}\left[ L_{x},L_{y}\right] \right)%
	\text{.}
\end{equation*}

Now we check that%
\begin{equation}
	\left\{ \left( e_{s},\lambda L_{e_{s}}\right) \mid 1\leq s\leq 7\right\}
	\cup \left\{ \left( 2e_{i}\times e_{j},\lambda \left[ L_{e_{i}},L_{e_{j}}%
	\right] \right) \mid 1\leq i<j\leq 7\right\}  \label{union}
\end{equation}%
is a basis of $\mathbb{R}^{7}\rtimes \mathfrak{o}\left( 7\right) $. We set
up the equation 
\begin{equation}
	\sum_{s=1}^{7}a_{s}\left( e_{s},\lambda L_{e_{s}}\right) +\sum_{1\leq
		i<j\leq 7}a_{ij}\left( 2e_{i}\times e_{j},\lambda \left[ L_{e_{i}},L_{e_{j}}%
	\right] \right) =0\text{.}  \label{li1}
\end{equation}

For each $s=1,\dots ,7$ let $\iota \left( s\right) =\left\{ \left(
i,j\right) \in \left\{ 1,\dots ,7\right\} ^{2}\mid e_{i}\times
e_{j}=e_{s}\right\} $, whose elements can be read off from the following
table, built up easily using (\ref{multiplicaO}):%
\begin{equation}
	\begin{tabular}{ll}
		$e_{1}=e_{5}\times e_{6}=e_{2}\times e_{4}=e_{3}\times e_{7}$ & $%
		e_{2}=e_{6}\times e_{7}=e_{4}\times e_{1}=e_{3}\times e_{5}$ \\ 
		$e_{3}=e_{7}\times e_{1}=e_{5}\times e_{2}=e_{4}\times e_{6}$ & $%
		e_{4}=e_{1}\times e_{2}=e_{6}\times e_{3}=e_{5}\times e_{7}$ \\ 
		$e_{5}=e_{2}\times e_{3}=e_{7}\times e_{4}=e_{6}\times e_{1}$ & $%
		e_{6}=e_{3}\times e_{4}=e_{7}\times e_{2}=e_{1}\times e_{5}$ \\ 
		$e_{7}=e_{4}\times e_{5}=e_{2}\times e_{6}=e_{1}\times e_{3}$ & 
	\end{tabular}
	\label{table124}
\end{equation}

For $\left( i,j\right) \in \iota \left( s\right) $ we have from (\ref%
{decompo}) that%
\begin{equation*}
	\left[ L_{e_{i}},L_{e_{j}}\right] =-L_{e_{s}}+Z\left( e_{i},e_{j}\right) 
	\text{,}
\end{equation*}%
with $Z\left( e_{i},e_{j}\right) =3e_{i}\wedge e_{j}-L_{e_{s}}\in \mathfrak{g%
}_{2}$, as in (\ref{Zxy}). We can rewrite (\ref{li1}) as%
\begin{equation}
	\sum_{s=1}^{7}a_{s}\left( e_{s},\lambda L_{e_{s}}\right)
	+\sum_{s=1}^{7}\sum_{\left( i,j\right) \in \iota \left( s\right)
	}a_{s}\left( i,j\right) \left( 2e_{s},\lambda \left(
	2L_{e_{s}}-3L_{e_{s}}+Z\left( e_{i},e_{j}\right) \right) \right) =0\text{.}
	\label{SumSum}
\end{equation}

Now, using (\ref{decom}), we see that 
\begin{equation}
	\mathbb{R}^{7}\rtimes \mathfrak{0}\left( 7\right) =\mathcal{D}^{\lambda
	}\oplus \mathcal{V}\oplus \left( \left\{ 0\right\} \times \mathfrak{g}%
	_{2}\right) \text{,}  \label{decomG}
\end{equation}%
where $\mathcal{V}=\left\{ \left( 0,L_{y}\right) \mid y\in \mathbb{R}%
^{7}\right\} $. Hence, all three components of the left hand side of (\ref%
{SumSum}) with respect to the decomposition vanish separately. Thus, calling 
$\sigma \left( s\right) =\sum_{\left( i,j\right) \in \iota \left( s\right)
}a_{s}\left( i,j\right) $, we have%
\begin{equation*}
	\sum_{s=1}^{7}\left( a_{s}+2\sigma \left( s\right) \right) \left(
	e_{s},\lambda L_{e_{s}}\right) =0\text{,\ \ \ \ \ \ }\sum_{s=1}^{7}\sigma
	\left( s\right) \left( 0,L_{e_{s}}\right) =0
\end{equation*}%
and hence $\sigma \left( s\right) =0=a_{s}$ for all $s$. Next we consider
the component in $\left\{ 0\right\} \times \mathfrak{g}_{2}$. By the
definitions of $Z\left( e_{i},e_{j}\right) $ and $\sigma \left( s\right) $
we have 
\begin{equation*}
	0=\sum_{s=1}^{7}\sum_{\left( i,j\right) \in \iota \left( s\right)
	}a_{s}\left( i,j\right) Z\left( e_{i},e_{j}\right)
	=3\sum_{s=1}^{7}\sum_{\left( i,j\right) \in \iota \left( s\right)
	}a_{s}\left( i,j\right) e_{i}\wedge e_{j}-\sum_{s=1}^{7}\sigma \left(
	s\right) L_{e_{s}}\text{.}
\end{equation*}

Now, the last term equals zero. Also, by (\ref{table124}), $\left\{
e_{i}\wedge e_{j}\mid \left( i,j\right) \in \iota \left( s\right) \text{, }%
s=1,\dots ,7\right\} $ is linearly independent. Consequently, $a_{s}\left(
i,j\right) =0$ for all $\left( i,j\right) \in \iota \left( s\right) $ and
all $s$ and so (\ref{union}) is a basis of $\mathbb{R}^{7}\rtimes \mathfrak{o%
}\left( 7\right) $, as desired. \hfill $\square $

\subsection{Sub-Riemannian geodesics as orbits. A second criterion}

In the proof of Theorem \ref{OctoGeode} we will use conditions for a
sub-Riemannian geodesic to be the orbit of a monoparametric subgroup of a
Lie group acting on the manifold.

We state the criterion in Proposition \ref{Toth} below. In order to do that,
it is convenient to present the sub-Riemannian structure $g$ on $\left( N,\mathcal{D}%
,g\right) $ in an equivalent manner as a cometric, that is, a morphism of
vector bundles $b:TN^{\ast }\rightarrow TN$ with $b^{\ast }=b$ and $\alpha
\left( b\left( \alpha \right) \right) >0$ for all $\alpha \in T^{\ast }N$.
The relationship is as follows: The distribution $\mathcal{D}$ is the image
of $b$ and 
\begin{equation}
	g\left( u,v\right) =\beta \left( u\right) =\alpha \left( v\right) \text{\ \
		\ \ \ if }v=b\left( \beta \right) \text{ and\ }u=b\left( \alpha \right) .
	\label{subR*}
\end{equation}%
The Hamiltonian $H$ of $\left( N,b\right) $ is defined by 
\begin{equation}
	H:T^{\ast }N\rightarrow \mathbb{R}\text{, \ \ \ \ \ \ \ \ \ }H\left( \alpha
	\right) =\tfrac{1}{2}\alpha \left( b\left( \alpha \right) \right) .
	\label{Hamil}
\end{equation}

We recall Lemma 3.3 in \cite{Toth}, that will allow us to obtain geodesics
of our system. See Lemma 3.4 in \cite{AP} for an alternative proof.

\begin{proposition}
	\label{Toth}\cite{Toth, AP} Let $\bar{G}$ be a Lie group, with Lie algebra $%
	\mathfrak{\bar{g}}$, acting transitively on a smooth manifold $N$ and let $%
	o\in N$. Let $f:\mathfrak{\bar{g}}\rightarrow T_{o}N$ be the linear map
	defined by 
	\begin{equation*}
		f\left( Y\right) =\left. \tfrac{d}{dt}\right\vert _{0}\exp \left( tY\right)
		\cdot o\text{.}
	\end{equation*}
	
	Let $H:T^{\ast }N\rightarrow \mathbb{R}$ be a $\bar{G}$-invariant
	Hamiltonian function and let $H_{o}=\left. H\right\vert _{T_{o}^{\ast
		}N}:T_{o}^{\ast }N\rightarrow \mathbb{R}$.
	
	Let $Y\in \mathfrak{\bar{g}}$ and let $\alpha \in T_{o}^{\ast }N$. Then the
	curve $t\mapsto \exp \left( tY\right) \cdot o$ is the normal geodesic
	associated with the momentum $\alpha $ if and only if%
	\begin{equation}
		\left( dH_{o}\right) _{\alpha }=f\left( Y\right) \text{\ \ \ \ \ \ \ and\ \
			\ \ \ \ \ }f^{\ast }\left( \alpha \right) \circ \text{\emph{ad}}_{Y}=0\text{.%
		}  \label{TothDisplay}
	\end{equation}
\end{proposition}

Notice that $\left( dH_{o}\right) _{\alpha }:T_{\alpha }\left( T_{o}^{\ast
}N\right) \equiv T_{o}^{\ast }N\rightarrow \mathbb{R}$ is a linear map an
so, $\left( dH_{o}\right) _{\alpha }\in \left( T_{o}^{\ast }N\right) ^{\ast
}=T_{o}N$. Also, given $\alpha \in T_{o}^{\ast }N$, the set $\left\{ Y\in 
\mathfrak{\bar{g}}\mid Y\text{ satisfies (\ref{TothDisplay})}\right\} $ is
an affine subspace and any two vectors there give rise to the same
sub-Riemannian geodesic.

\medskip

Next we state the particular case relevant in our situation.

\begin{corollary}
	\label{CoroToth}Let $K$ be a closed subgroup of a Lie group $G$, with Lie
	algebras $\mathfrak{k}$ and $\mathfrak{g}$, respectively. Let $b:T^{\ast
	}G\rightarrow TG$ be a sub-Riemannian structure on $G$ invariant by the
	action of $\bar{G}=G\times K$ on $G$ given by $\left( g,k\right) \cdot
	h=ghk^{-1}$.
	
	\smallskip
	
	Let $\alpha \in T_{e}^{\ast }G$ and $X\in \mathfrak{g}$, $Z\in \mathfrak{k}$%
	. Then the three identities 
	\begin{equation}
		b\left( \alpha \right) =X-Z\text{,\ \ \ \ \ \ \ }\alpha \circ \text{\emph{ad}%
		}_{X}=0\in \mathfrak{g}^{\ast }\text{\ \ \ \ \ \ and \ \ \ \ \ }\alpha \circ 
		\text{\emph{ad}}_{Z}=0\in \mathfrak{k}^{\ast }  \label{CoroTothDisplay}
	\end{equation}%
	are satisfied if and only if $\gamma \left( t\right) =\exp \left( tX\right)
	\exp \left( -tZ\right) $ is the unique normal sub-Riemannian geodesic in $G$
	with initial momentum $\alpha $. Its initial velocity equals $X-Z$.
\end{corollary}

\begin{proof}
	In Proposition \ref{Toth} consider $\bar{G}=G\times K$, $N=G$ and $Y=\left(
	X,Z\right) $. The $\bar{G}$-invariance of the Hamiltonian follows from $\bar{%
		G}$-invariance of $b$. By (\ref{Hamil}) we have 
	\begin{eqnarray*}
		\left( dH_{o}\right) _{\alpha }\left( \beta \right) &=&\left. \tfrac{d}{dt}%
		\right\vert _{0}H_{o}\left( \alpha +t\beta \right) =\left. \tfrac{d}{dt}%
		\right\vert _{0}\tfrac{1}{2}\left( \alpha +t\beta \right) \left( b\left(
		\alpha +t\beta \right) \right) \\
		&=&\left. \tfrac{d}{dt}\right\vert _{0}\tfrac{t}{2}\left( \beta \left(
		b\left( \alpha \right) \right) +\alpha \left( b\left( \beta \right) \right)
		\right) =\beta \left( b\left( \alpha \right) \right)
	\end{eqnarray*}%
	(the last equality holds since $b^{\ast }=b$). So, $\left( dH_{o}\right)
	_{\alpha }=b\left( \alpha \right) $. Now we compute $f:\mathfrak{\bar{g}%
		=g\times k}\rightarrow T_{e}G=\mathfrak{g}$: 
	\begin{equation}
		f\left( X,Z\right) =\left. \tfrac{d}{dt}\right\vert _{0}\exp \left( t\left(
		X,Z\right) \right) e=\left. \tfrac{d}{dt}\right\vert _{0}\exp \left(
		tX\right) \exp \left( -tZ\right) =X-Z\text{.}  \label{f(x,z)}
	\end{equation}
		Thus, the first equation in (\ref{TothDisplay}) translates into the first
	equation in (\ref{CoroTothDisplay}).
	
	By (\ref{f(x,z)}), $\left( f^{\ast }\left( \alpha \right) \right) \left(
	X,Z\right) =\alpha \left( f\left( X,Z\right) \right) =\alpha \left(
	X-Z\right) $ for $\alpha \in \mathfrak{g}^{\ast }$ and so, 
	\begin{equation*}
		\left( f^{\ast }\left( \alpha \right) \right) \left[ \left( X,Z\right)
		,\left( X^{\prime },X^{\prime }\right) \right] =\alpha \left( \left[
		X,X^{\prime }\right] ,\left[ Z,Z^{\prime }\right] \right) =\alpha \left( %
		\left[ X,X^{\prime }\right] \right) -\alpha \left( \left[ Z,Z^{\prime }%
		\right] \right)
	\end{equation*}%
	for all $X,X^{\prime }\in \mathfrak{g}$, $Z,Z^{\prime }\in \mathfrak{k}$.
	Thus, the second equation in (\ref{TothDisplay}) translates into the second
	and third equations in (\ref{CoroTothDisplay}). The initial velocity of $%
	\gamma $ is $\gamma ^{\prime }\left( 0\right) =\left. \frac{d}{dt}%
	\right\vert _{0}\exp \left( tx\right) +\left. \frac{d}{dt}\right\vert
	_{0}\exp \left( -tz\right) =x-z$.
\end{proof}


\subsection{Sub-Riemannian geodesics of the octonionic system}

We present the proof of the result which gives explicitly geodesics of $%
\left( \mathbb{R}^{7}\rtimes SO_{7},\mathcal{D}^{\lambda }\right) $.

\medskip

\noindent \textit{Proof of Theorem \ref{OctoGeode}. }We may suppose that $%
\left\Vert x\right\Vert =1$, since if $\gamma $ is the geodesic with initial
momentum $\theta \in \mathfrak{g}^{\ast }$, then the curve $\sigma $ given
by $\sigma \left( t\right) =\gamma \left( ct\right) $ is the geodesic with
initial momentum $c\theta $ (notice that $Z$ is bilinear in $x$ and $y$).

We suppose first that $\left\Vert y\right\Vert =:n\neq 0$ and call $x_{1}=x$%
, $x_{2}=y/n$ and $x_{3}=x_{1}\times x_{2}$. Let $\left\{ x_{i}\mid
i=1,\dots ,7\right\} $ be a basis of $\mathbb{R}^{7}$ and let $\mathcal{G}%
=\left\{ Z_{1},\dots ,Z_{14}\right\} $ be any basis of $\left\{ 0\right\}
\times \mathfrak{g}_{2}$. Let $\mathcal{B}$ be the juxtaposition of the sets 
$\left\{ \left( x_{i},\lambda L_{x_{i}}\right) \mid i=1,\dots ,7\right\} $, $%
\left\{ \left( 0,L_{x_{i}}\right) \mid i=1,\dots ,7\right\} $ and $\mathcal{G%
}$, which together form a basis of $\mathfrak{g}:=\mathbb{R}^{7}\rtimes 
\mathfrak{o}\left( 7\right) $ by (\ref{decomG}), and let $\left\{ \delta
_{1},\dots ,\delta _{7},\nu _{1},\dots ,\nu _{7},\zeta _{1},\dots ,\zeta
_{14}\right\} $ be the basis of $\mathfrak{g}^{\ast }$ dual to $\mathcal{B}$.

Define $b:\mathfrak{g}^{\ast }\rightarrow \mathfrak{g}$ by $b\left( \delta
_{i}\right) =\left( x_{i},\lambda L_{x_{i}}\right) $, $b\left( \nu
_{i}\right) =0=b\left( \zeta _{j}\right) $ for $i=1,\dots ,7$ and $j=1,\dots
,14$. The map $b$ is the cometric corresponding to the sub-Riemannian metric
we are considering. Indeed, it is not difficult to see that $b^{\ast }=b$
and the image of $b$ is $\mathcal{D}^{\lambda }$; also,%
\begin{equation*}
	\left\langle x_{i},x_{j}\right\rangle =\left\langle \left( x_{i},\lambda
	L_{x_{i}}\right) ,\left( x_{j},\lambda L_{x_{j}}\right) \right\rangle
	=\delta _{j}\left( x_{i},\lambda L_{x_{i}}\right) =\delta _{ij}\text{,}
\end{equation*}%
as required (see (\ref{subR*})).

We apply Corollary \ref{CoroToth} with 
\begin{equation*}
	\mathfrak{\bar{g}=g}\times \mathfrak{g}_{2}\text{,\ \ \ \ \ \ }X=\left(
	x,\lambda L_{x}+Z\left( x,y\right) \right) \text{\ \ \ \ \ \ \ and\ \ \ \ \
		\ \ \ \ }Z=Z\left( x,y\right) 
\end{equation*}%
to prove that $\gamma _{x,y}$ is the sub-Riemannian geodesic with initial
momentum $\alpha =\delta _{1}+c\nu _{1}+d\nu _{3}$ for some numbers $c,d$ to
be determined later. We compute%
\begin{equation*}
	b\left( \alpha \right) =b\left( \delta _{1}+c\nu _{1}+d\nu _{3}\right)
	=b\left( \delta _{1}\right) =\left( x,\lambda L_{x}\right) =\left( x,\lambda
	L_{x}+Z\left( x,y\right) \right) -\left( 0,Z\left( x,y\right) \right) \text{.%
	}
\end{equation*}%
So, the first identity in (\ref{CoroTothDisplay}) holds.

Since $\mathfrak{g}_{2}$ is a subalgebra of $\mathfrak{o}\left( 7\right) $, $%
\alpha \left( 0,\left[ Z\left( x,y\right) ,W\right] \right) =0$ for all $%
W\in \mathfrak{g}_{2}$. This yields the third identity in (\ref%
{CoroTothDisplay}). Now we verify the second one in our case.

By (\ref{decomG}), an arbitrary element $X^{\prime }$ of $\mathfrak{g}=%
\mathbb{R}^{7}\rtimes \mathfrak{o}\left( 7\right) $ can be written as 
\begin{equation*}
	X^{\prime }=\left( u,\lambda L_{u}\right) +\left( 0,L_{v}\right) +\left(
	0,W\right)
\end{equation*}%
with $u,v\in \mathbb{R}^{7}$ and $W\in \mathfrak{g}_{2}$. Next we compute $%
\alpha \left( \left[ X,X^{\prime }\right] \right) $.

Using the expression $\left[ \left( a,A\right) ,\left( b,B\right) \right]
=\left( Ab-Ba,\left[ A,B\right] \right) $ for the Lie bracket on $\mathbb{R}%
^{7}\rtimes \mathfrak{o}\left( 7\right) $, (\ref{decompo}) and (\ref{corZL}%
), we obtain that $\left[ X,X^{\prime }\right] $ is the sum of the following
six terms: 
\begin{eqnarray*}
	\text{1) }\left[ \left( x,\lambda L_{x}\right) ,\left( u,\lambda
	L_{u}\right) \right] &=&\left( 2\lambda x\times u,\lambda ^{2}\left[
	L_{x},L_{u}\right] \right) \\
	&=&\left( 2\lambda x\times u,\lambda ^{2}Z\left( x,u\right) -\lambda
	^{2}L_{x\times u}\right) \\
	&=&2\lambda \left( x\times u,\lambda L_{x\times u}\right) -3\lambda
	^{2}\left( 0,L_{x\times u}\right) +\lambda ^{2}\left( 0,Z\left( x,u\right)
	\right) \text{.}
\end{eqnarray*}%
\begin{eqnarray*}
	\text{2) }\left[ \left( x,\lambda L_{x}\right) ,\left( 0,L_{v}\right) \right]
	&=&\left( -L_{v}\left( x\right) ,\lambda \left[ L_{x},L_{v}\right] \right)
	=\left( x\times v,\lambda \left( Z\left( x,v\right) -L_{x\times v}\right)
	\right) \\
	&=&\left( x\times v,\lambda L_{x\times v}\right) -2\lambda \left(
	0,L_{x\times v}\right) +\lambda \left( 0,Z\left( x,v\right) \right) .
\end{eqnarray*}%
\begin{equation*}
	\begin{tabular}{l}
		3) $\left[ \left( x,\lambda L_{x}\right) ,\left( 0,W\right) \right] =-\left(
		W\left( x\right) ,\lambda L_{W\left( x\right) }\right) $. \\ 
		4) $\left[ \left( 0,Z\left( x,y\right) \right) ,\left( u,\lambda
		L_{u}\right) \right] =\left( Z\left( x,y\right) \left( u\right) ,\lambda
		L_{Z\left( x,y\right) \left( u\right) }\right) $. \\ 
		5) $\left[ \left( 0,Z\left( x,y\right) \right) ,\left( 0,L_{v}\right) \right]
		=\left( 0,L_{Z\left( x,y\right) \left( v\right) }\right) $. \\ 
		6) $\left[ \left( 0,Z\left( x,y\right) \right) ,\left( 0,W\right) \right]
		=\left( 0,\left[ Z\left( x,y\right) ,W\right] \right) $.%
	\end{tabular}%
\end{equation*}

Putting $z=\left\langle z,x\right\rangle x+z^{\prime }$ with $z^{\prime
}\perp x$, one sees that $\delta _{1}\left( z,\lambda L_{z}\right)
=\left\langle z,x\right\rangle $. Then, as $x\times u$, $x\times v$ and $%
W(x) $ are orthogonal to $x$ ($W$ is skew-symmetric), we have that%
\begin{equation*}
	\delta _{1}\left( \left[ X,X^{\prime }\right] \right) =\delta _{1}\left(
	Z\left( x,y\right) \left( u\right) ,\lambda L_{Z\left( x,y\right) \left(
		u\right) }\right) =\left\langle Z\left( x,y\right) \left( u\right)
	,x\right\rangle =-\left\langle Z\left( x,y\right) \left( x\right)
	,u\right\rangle \text{.}
\end{equation*}%
Similarly, 
\begin{eqnarray*}
	\nu _{1}\left( \left[ X,X^{\prime }\right] \right) &=&\left\langle Z\left(
	x,y\right) \left( v\right) ,x\right\rangle =-\left\langle Z\left( x,y\right)
	\left( x\right) ,v\right\rangle \text{,} \\
	\nu _{3}\left( \left[ X,X^{\prime }\right] \right) &=&-3\lambda ^{2}\nu
	_{3}\left( 0,L_{x\times u}\right) -2\lambda \nu _{3}\left( 0,L_{x\times
		v}\right) +\nu _{3}\left( 0,L_{Z\left( x,y\right) \left( v\right) }\right) \\
	&=&-3\lambda ^{2}\left\langle x\times u,x_{3}\right\rangle -2\lambda
	\left\langle x\times v,x_{3}\right\rangle +\left\langle Z\left( x,y\right)
	\left( v\right) ,x_{3}\right\rangle \\
	&=&3\lambda ^{2}\tfrac{1}{n}\left\langle x\times \left( x\times y\right)
	,u\right\rangle +2\lambda \tfrac{1}{n}\left\langle x\times \left( x\times
	y\right) ,v\right\rangle -\tfrac{1}{n}\left\langle Z\left( x,y\right) \left(
	x\times y\right) ,v\right\rangle \\
	&=&-3\lambda ^{2}\tfrac{1}{n}\left\langle y,u\right\rangle -\tfrac{1}{n}%
	\left\langle 2\lambda y+Z\left( x,y\right) \left( x\times y\right)
	,v\right\rangle \text{.}
\end{eqnarray*}

Therefore, $\alpha \left( \left[ X,X^{\prime }\right] \right) =\left( \delta
_{1}+c\nu _{1}+d\nu _{3}\right) \left( \left[ X,X^{\prime }\right] \right) $
vanishes for all $X^{\prime }$ if an only if%
\begin{equation}
	-nZ\left( x,y\right) \left( x\right) -3\lambda ^{2}dy=0\text{\ \ \ \ \ \ \ \
		and\ \ \ \ \ \ }-cnZ\left( x,y\right) \left( x\right) -2\lambda dy-dZ\left(
	x,y\right) \left( x\times y\right) =0 \text{.} \label{finalEq}
\end{equation}%
We compute%
\begin{eqnarray*}
	Z\left( x,y\right) \left( x\right)  &=&3\left( x\wedge y\right) \left(
	x\right) -\left( x\times y\right) \times x=3\left( \left\langle
	x,x\right\rangle y-\left\langle x,y\right\rangle x\right) -y=2y\text{,} \\
	Z\left( x,y\right) \left( x\times y\right)  &=&3\left( x\wedge y\right)
	\left( x\times y\right) -\left( x\times y\right) \left( x\times y\right) =0%
	\text{.}
\end{eqnarray*}%
Hence, a straightforward computation yields that (\ref{finalEq}) is equivalent
to $3\lambda c=2$ and $3\lambda ^{2}d=-2n$. Solving for $c$ and $d$, we
obtain that $\gamma _{x,y}$ is the geodesic with initial momentum $\alpha $.

If $n=0$, that is, $y=0$, similar, but much simpler arguments apply to show
that $\gamma _{x,0}$ is the geodesic with initial momentum $\alpha _{1}$.
\hfill $\square $

\begin{proposition}
	\label{momentumO}The sub-Riemannian geodesics in the theorem above are
	exactly those with initial momentum vanishing at $\left\{ 0\right\} \times 
	\mathfrak{g}_{2}$.
\end{proposition}

\begin{proof}
	By the proof of the theorem, $\gamma _{x,y}$ satisfies the condition.
	Conversely, let $\alpha \in \mathfrak{g}^{\ast }$ with $\alpha \left( \left\{
	0\right\} \times \mathfrak{g}_{2}\right) =0$. Recalling (\ref{decomG}), call 
	$\phi :\mathbb{R}^{7}\rightarrow \mathcal{D}^{\lambda }$, $\phi \left(
	z\right) =\left( z,\lambda L_{x}\right) $ and $\psi :\mathbb{R}%
	^{7}\rightarrow \mathcal{V}$, $\psi \left( z\right) =\left( 0,L_{z}\right) $%
	. Completing in a suitable way a basis of Ker$~\left( \alpha \circ \phi
	\right) \cap ~$Ker~$\left( \alpha \circ \psi \right) $, one can construct an
	orthonormal basis $\left\{ x_{1},\dots ,x\right\} $ of $\mathbb{R}^{7}$ such
	that $\alpha \left( x_{i},\lambda L_{x_{i}}\right) =0$ for $i>1$ and $\alpha
	\left( 0,L_{x_{i}}\right) =0$ for $i>2$. The statement follows now from the
	proof of the theorem.
\end{proof}

\section{Appendix. Friendlier presentations of $K^{\mathbb{C}}/K$ for
	classical groups\label{appendix}}

We take the opportunity to give a (for us) more amiable and concrete
presentation of the symmetric spaces $M=K^{\mathbb{C}}/K$ in the second row
of Table (\ref{table}) when $K\ $is a classical group, following \cite%
{Neretin}. They are%
\begin{equation}
	SO\left( n,\mathbb{C}\right) /SO\left( n\right) \text{,\ \ \ \ }SL\left( n,%
	\mathbb{C}\right) /SU\left( n\right) \text{\ \ \ \ \ and\ \ \ \ \ }Sp\left(
	2n,\mathbb{C}\right) /Sp\left( n\right) \text{.}  \label{3quo}
\end{equation}

Let $\mathbb{F}=\mathbb{R}$, $\mathbb{C}$ or the quaternions $\mathbb{H}$
and consider on $\mathbb{F}^{n}$ the Hermitian inner product $\left\langle
x,y\right\rangle =x^{\ast }y$, where $\ast $ means conjugate transpose ($%
\mathbb{H}^{n}$ as a right vector space over $\mathbb{H}$). Let 
\begin{equation*}
	U\left( n,\mathbb{F}\right) =\left\{ A\in \mathbb{F}^{n\times n}\mid A^{\ast
	}A=I_{n}\right\}
\end{equation*}%
be the group of $\mathbb{F}$-linear isometries of $\mathbb{F}^{n}$. We have 
\begin{equation*}
	U\left( n,\mathbb{R}\right) =O\left( n\right) \text{,\ \ \ \ \ \ \ }U\left(
	n,\mathbb{C}\right) =U\left( n\right) \text{\ \ \ \ \ \ \ and\ \ \ \ \ \ \ }%
	U\left( n,\mathbb{H}\right) =Sp\left( n\right) \text{.}
\end{equation*}

The quotients $M$ in (\ref{3quo}) are canonically isomorphic to $U\left( n,%
\mathbb{F}\right) ^{\mathbb{C}}/U\left( n,\mathbb{F}\right) $ (except for $%
\mathbb{F}=\mathbb{C}$, in which case $M$ is a hypersurface of $GL\left( n,%
\mathbb{C}\right) /U\left( n\right) $).

\begin{proposition}
	\label{amiable}The quotient $U\left( n,\mathbb{F}\right) ^{\mathbb{C}%
	}/U\left( n,\mathbb{F}\right) $ can be identified naturally with 
	\begin{equation*}
		U^{+}\left( n,\mathbb{F}\right) :=\left\{ A\in U\left( n,\mathbb{F}\right)
		\mid \text{\emph{Re}~}\left( \mu \right) >0\text{ for any eigenvalue }\mu 
		\text{ of }A\right\} \text{.}
	\end{equation*}
\end{proposition}

\begin{remark}
	One has that $\mathbb{F}$-linear isometries in $U^{+}\left( n,\mathbb{F}%
	\right) $ rotate planes in $\mathbb{F}^{n}$ through angles $\theta $ with $%
	\left\vert \theta \right\vert <\pi /2$.\thinspace If $U^{+}\left( n,\mathbb{F%
	}\right) $ is endowed with the symmetric Riemannian metric invariant by the
	action of $U\left( n,\mathbb{F}\right) ^{\mathbb{C}}$, no element of it is
	distinguished \emph{(}the space is of course homogeneous\emph{)} and
	rotations through the angle $\pi /2$ are \textquotedblleft at
	infinity\textquotedblright . This allows us to think of $U^{+}\left( n,%
	\mathbb{F}\right) $, informally, as the set of \textquotedblleft small
	rotations\textquotedblright\ of $\mathbb{F}^{n}$.
\end{remark}

\smallskip

Examples 5, 14 and 30 in List 1 in \cite{Neretin} describe the quotient $%
U\left( n,\mathbb{F}\right) ^{\mathbb{C}}/U\left( n,\mathbb{F}\right) $ as
the Grassmannian $\mathcal{G}_{0}\left( n,n\right) $ of maximal isotropic
subspaces of $\mathbb{F}^{n,n}$, that is, $\mathbb{F}^{2n}$ endowed with the
split Hermitian inner product \thinspace 
\begin{equation*}
	g\left( \left( x,y\right) ,\left( u,v\right) \right) =x^{\ast }u-y^{\ast }v
\end{equation*}%
(recall that an $\mathbb{F}$-subspace $V$ of $\mathbb{F}^{n,n}$ is said to
be isotropic if $g\left( X,Y\right) =0$ for all $X,Y$; for its properties,
see for instance \cite{Harvey}). Applying the arguments in Section 3.4 of 
\cite{Neretin}, we get the desired identification of a connected component
of $\mathcal{G}_{0}\left( n,n\right) $ with $U^{+}\left( n,\mathbb{F}\right) 
$.

Next we particularize to our situation and expand the concise reasoning of
Neretin's article, which has a much broader scope.

\bigskip

\noindent\textit{Proof of Proposition \ref{amiable}.} First we sketch the
steps of the proof and below we elaborate on each one of them.

\textsl{Step 1}. We identify the unitary transformations $A$ in $U\left( n,%
\mathbb{F}\right) $ with their (reflected) graphs $\left\{ \left(
Ax,x\right) \mid x\in \mathbb{F}\right\} $, which turn out to be exactly the
maximal isotropic subspaces of $\mathbb{F}^{n,n}$. More precisely, the map%
\begin{equation}
	F:U\left( n,\mathbb{F}\right) \rightarrow \mathcal{G}_{0}\left( n,n\right) 
	\text{,\ \ \ \ \ \ \ \ }F\left( A\right) =\left\{ \left( Ax,x\right) \mid
	x\in \mathbb{F}\right\} \text{,}  \label{Fgraph}
\end{equation}%
is a bijection. The group $U\left( n,n,\mathbb{F}\right) $ of all the $%
\mathbb{F}$-linear isometries of $\mathbb{F}^{n,n}$ acts naturally on $%
\mathcal{G}_{0}\left( n,n\right) $ and the action is transitive.

\smallskip

\textsl{Step 2}. There is a natural isomorphism 
\begin{equation}
	\psi :U^{J}\left( n,n,\mathbb{F}\right) :=\left\{ A\in U\left( n,n,\mathbb{F}%
	\right) \mid AJ=JA\right\} \rightarrow U\left( n,\mathbb{F}\right) ^{\mathbb{%
			C}}\text{,}  \label{psi}
\end{equation}%
where $J:\mathbb{F}^{n,n}\rightarrow \mathbb{F}^{n,n}$ is given by $J\left(
x,y\right) =\left( -y,x\right) $.

\smallskip

\textsl{Step 3}. The restriction to $U^{J}\left( n,n,\mathbb{F}\right) $ of
the action of $U\left( n,n,\mathbb{F}\right) $ on $\mathcal{G}_{0}\left(
n,n\right) $ preserves%
\begin{equation*}
	\mathcal{G}_{0}^{J}\left( n,n\right) :=\left\{ V\in \mathcal{G}_{0}\left(
	n,n\right) \mid J\left( V\right) \cap V=\left\{ 0\right\} \right\}
\end{equation*}%
and is transitive on each connected component of it, with isotropy subgroup
at $V_{o}=\left\{ \left( x,x\right) \mid x\in \mathbb{F}\right\} $ equal to $%
U\left( n,\mathbb{F}\right) $.

\smallskip

\textsl{Step 4}. The preimage of $\mathcal{G}_{0}^{J}\left( n,n\right) $
under the map $F$ in (\ref{Fgraph}) is 
\begin{equation*}
	U^{\prime }\left( n,\mathbb{F}\right) :=\left\{ A\in U\left( n,\mathbb{F}%
	\right) \mid A^{2}+I\text{ is not singular}\right\}
\end{equation*}%
and $F\left( I_{n}\right) =V_{o}$. Finally, $U^{+}\left( n,\mathbb{F}\right) 
$ is the identity component of $U^{\prime }\left( n,\mathbb{F}\right) $.

The following commutative diagram could be helpful.%
\begin{equation*}
	\begin{array}{ccc}
		U\left( n,\mathbb{F}\right) ^{\mathbb{C}}\times U^{\prime }\left( n,\mathbb{F%
		}\right) & \longrightarrow & U^{\prime }\left( n,\mathbb{F}\right) \\ 
		&  &  \\ 
		\ \ \ \ \ \ \ \ \downarrow \left( \psi ^{-1},F\right) &  & \ \downarrow F \\ 
		&  &  \\ 
		U^{J}\left( n,n,\mathbb{F}\right) \times \mathcal{G}_{0}^{J}\left( n,n\right)
		& \longrightarrow & \mathcal{G}_{0}^{J}\left( n,n\right)%
	\end{array}%
\end{equation*}

Next we discuss each point in detail.

\smallskip

\textsl{Step 1.} We compute 
\begin{equation}
	g\left( \left( Ax,x\right) ,\left( Ay,y\right) \right) =\left( Ax\right)
	^{\ast }Ay-x^{\ast }y=x^{\ast }\left( A^{\ast }A-I_{n}\right) y.
	\label{CuentaNUll}
\end{equation}%
Thus, $F\left( A\right) $ is an isotropic subspace of $\mathbb{F}^{n,n}$,
which is maximal since it has dimension $n$.

A maximal isotropic subspace projects isomorphically to $\left\{ 0\right\}
\times \mathbb{F}^{n}$. Hence, it has the form $\left\{ \left( Bx,x\right)
\mid x\in \mathbb{F}^{n}\right\} $ for some $\mathbb{F}$-linear operator $B$%
. By (\ref{CuentaNUll}), $B\in U\left( n,\mathbb{F}\right) $. Therefore, the
image of $F$ is $\mathcal{G}_{0}\left( n,n\right) $.

The action of $U\left( n,n,\mathbb{F}\right) $ on $\mathcal{G}_{0}\left(
n,n\right) $ is the natural one. It is transitive by Witt's Theorem. The
induced action on $U\left( n,\mathbb{F}\right) $ through $F$ in (\ref{Fgraph}%
) is by M\"{o}bius transformations: If 
\begin{equation*}
	X=\left( 
	\begin{array}{cc}
		a & c \\ 
		b & d%
	\end{array}%
	\right) \in U\left( n,n,\mathbb{F}\right) \text{,\ \ \ \ \ then\ \ \ \ \ }%
	X\cdot A=\left( aA+c\right) \left( bA+d\right) ^{-1}\text{,}
\end{equation*}%
since for a suitable $y\in \mathbb{F}^{n\times n}$ one has%
\begin{equation}
	\left( 
	\begin{array}{cc}
		a & c \\ 
		b & d%
	\end{array}%
	\right) \left( 
	\begin{array}{c}
		Ax \\ 
		x%
	\end{array}%
	\right) =\left( 
	\begin{array}{c}
		aAx+c \\ 
		bAx+d%
	\end{array}%
	\right) =\left( 
	\begin{array}{c}
		\left( aA+c\right) \left( bA+d\right) ^{-1}y \\ 
		y%
	\end{array}%
	\right) \text{.}  \label{xy}
\end{equation}%
The third author has dealt with similar M\"{o}bius transformations in \cite%
{ESV}.

\smallskip

\textsl{Step 2.}\textbf{\ }The group $U^{J}\left( n,n,\mathbb{F}\right) $
consists of the all the matrices $X\in \mathbb{F}^{2n\times 2n}$ such that $%
X^{\ast }RX=X$ and $XJ=JX$, where 
\begin{equation*}
	R=\left( 
	\begin{array}{cc}
		I_{n} & 0_{n} \\ 
		0_{n} & -I_{n}%
	\end{array}%
	\right) \text{\ \ \ \ \ and\ \ \ \ \ }J=\left( 
	\begin{array}{cc}
		0 & -I_{n} \\ 
		I_{n} & 0%
	\end{array}%
	\right) \text{.}
\end{equation*}%
Straightforward computations yield%
\begin{eqnarray*}
	U^{J}\left( n,n,\mathbb{F}\right) &=&\left\{ \left( 
	\begin{array}{cc}
		a & -b \\ 
		b & a%
	\end{array}%
	\right) \mid 
	\begin{array}{l}
		a,b\in \mathbb{F}^{n\times n}\text{, }a^{\ast }a-b^{\ast }b=I_{n} \\ 
		\text{and }a^{\ast }b=-b^{\ast }a\text{.}%
	\end{array}%
	\right\} \\
	&=&\left\{ \left( 
	\begin{array}{cc}
		u\cos z & -u\sin z \\ 
		u\sin z & u\cos z%
	\end{array}%
	\right) \mid 
	\begin{array}{l}
		u,z\in \mathbb{F}^{n\times n}\text{, }u^{\ast }u=I_{n} \\ 
		\text{and }z^{\ast }=-z\text{.}%
	\end{array}%
	\right\} 
\end{eqnarray*}%
(here, as usual, $\cos z$ and $\sin z$ are defined by means of the
well-known power series).

Notice that for $\mathbb{F}=\mathbb{C}$, $\mathbb{H}$, we can have the more
common presentation 
\begin{equation*}
	\cos z=\cosh w\text{\ \ \ \ \ \ and\ \ \ \ \ \ }\sin z=-i\sinh w\text{,\ \ \
		\ \ \ with }w=iz
\end{equation*}%
($w$ is Hermitian symmetric if and only if $z$ is Hermitian skew-symmetric).

Define $\psi $ as in (\ref{psi}) by 
\begin{equation*}
	\psi \left( 
	\begin{array}{cc}
		a & -b \\ 
		b & a%
	\end{array}%
	\right) =a+ib=u\cos z+iu\sin z=u\left( \cos z+i\sin z\right) =u\exp \left(
	iz\right) \text{,}
\end{equation*}%
which is clearly one to one. It is also surjective, since 
\begin{equation*}
	U\left( n,\mathbb{F}\right) \times \left\{ z\in \mathbb{F}^{n\times n}\mid z%
	\text{ is skew-Hermitian}\right\} \rightarrow U\left( n,\mathbb{F}\right) ^{%
		\mathbb{C}},\ \ \ \left( u,z\right) \mapsto u\exp \left( iz\right)
\end{equation*}%
is the Cartan decomposition of $U\left( n,\mathbb{F}\right) ^{\mathbb{C}}$.

\smallskip

\textsl{Step 3.}\textbf{\ }If $V\in \mathcal{G}_{0}^{J}\left( n,n\right) $
and $X\in U^{J}\left( n,\mathbb{F}\right) $, then%
\begin{equation*}
	J\left( XV\right) \cap XV=X\left( JV\right) \cap XV=X\left( \left( JV\right)
	\cap V\right) =\left\{ 0\right\} \text{,}
\end{equation*}%
hence $X\left( V\right) \in \mathcal{G}_{0}^{J}\left( n,n\right) $. Note
that $g\left( J\left( X\right) ,J\left( Y\right) \right) =-g\left(
X,Y\right) $ for all $X,Y\in \mathbb{F}^{n,n}$. Proceeding as in (\ref{xy}),
we compute%
\begin{equation*}
	\left( 
	\begin{array}{cc}
		a & -b \\ 
		b & a%
	\end{array}%
	\right) \left( 
	\begin{array}{c}
		x \\ 
		x%
	\end{array}%
	\right) =\left( 
	\begin{array}{c}
		\left( a-b\right) x \\ 
		\left( b+a\right) x%
	\end{array}%
	\right)
\end{equation*}%
and see that $a-b=b+a$ if and only if $b=0$ and $a^{\ast }a=I_{n}$. Hence,
the isotropy subgroup at $V_{o}=\left\{ \left( x,x\right) \mid x\in \mathbb{F%
}\right\} $ is isomorphic to $U\left( n,\mathbb{F}\right) $. By a dimension
counting argument, the action is transitive on each connected component of $%
\mathcal{G}_{0}^{J}\left( n,n\right) $.

\smallskip

\textsl{Step 4.}\textbf{\ }Now we see that $\left\{ A\in U\left( n,\mathbb{F}%
\right) \mid F\left( A\right) \in \mathcal{G}_{0}^{J}\left( n,n\right)
\right\} =U^{\prime }\left( n,\mathbb{F}\right) $. We have that $\left(
Ax,x\right) =J\left( Ay,y\right) =\left( -y,Ay\right) $ if and only if $%
y=-Ax $ and $x=Ay$, which implies that $A^{2}y=-y$. This yields one
inclusion. For the remaining one, take $y\neq 0$ with $A^{2}y=-y$ and
consider $x=Ay$.

It is not difficult to verify that $A\in U^{\prime }\left( n,\mathbb{F}%
\right) $ if and only if $\pm i$ is not an eigenvalue of $A$ (for $\mathbb{F}%
=\mathbb{H}$ we mean the eigenvalues of the underlying $\mathbb{C}$-linear
map). The relationship between $U^{\prime }\left( n,\mathbb{F}\right) $ and $%
U^{+}\left( n,\mathbb{F}\right) $ follows from the (suitable stated)
continuity of the eigenvalues of a matrix as functions of its entries.
\hfill $\square $

\bigskip

\noindent \textbf{Example.\ }Let $K=SO\left( 2\right) $. Under the usual
identification of it with $S^{1}$ we have that $SO^{\prime }\left( 2\right)
\equiv \left\{ u\in S^{1}\mid \operatorname{Re}u\neq 0\right\} $. We make
explicit the action of $K^{\mathbb{C}}=SO\left( 2,\mathbb{C}\right) $ on $%
\varepsilon =\pm 1\in S^{1}$:%
\begin{equation}
	\left( 
	\begin{array}{cc}
		\cos \zeta & -\sin \zeta \\ 
		\sin \zeta & \cos \zeta%
	\end{array}%
	\right) \cdot \varepsilon =\varepsilon e^{-i\varepsilon \arcsin \left( \tanh
		2t\right) }  \label{arcsin}
\end{equation}%
if $\zeta =s+it\in \mathbb{C}$. Indeed, putting%
\begin{equation*}
	u=\left( 
	\begin{array}{cc}
		\cos s & -\sin s \\ 
		\sin s & \cos s%
	\end{array}%
	\right) \text{\ \ \ \ \ \ and\ \ \ \ \ \ }z=\left( 
	\begin{array}{cc}
		0 & -t \\ 
		t & 0%
	\end{array}%
	\right) \text{,}
\end{equation*}%
a straightforward computation using the power series of $\cos $ and $\sin $
yields%
\begin{equation*}
	\cos z+i\sin z=\left( 
	\begin{array}{cc}
		\cosh t & -i\sinh t \\ 
		i\sinh t & \cosh t%
	\end{array}%
	\right)
\end{equation*}%
and hence 
\begin{equation*}
	\psi \left( 
	\begin{array}{cc}
		u\cos z & -u\sin z \\ 
		u\sin z & u\cos z%
	\end{array}%
	\right) =u\left( \cos z+i\sin z\right) =\left( 
	\begin{array}{cc}
		\cos \zeta & -\sin \zeta \\ 
		\sin \zeta & \cos \zeta%
	\end{array}%
	\right) \text{.}
\end{equation*}

By the definition of the action of $SO^{J}\left( 2,2\right) $ on $SO\left(
2\right) $, as in (\ref{xy}), given $v\in SO\left( 2\right) $, we have 
\begin{equation*}
	\left( 
	\begin{array}{cc}
		u\cos z & -u\sin z \\ 
		u\sin z & u\cos z%
	\end{array}%
	\right) \left( 
	\begin{array}{c}
		vx \\ 
		x%
	\end{array}%
	\right) =\left( 
	\begin{array}{c}
		u\left( v\cos z-\sin z\right) x \\ 
		u\left( v\sin z+\cos z\right) x%
	\end{array}%
	\right) =\left( 
	\begin{array}{c}
		wy \\ 
		y%
	\end{array}%
	\right)
\end{equation*}%
for a suitable $y\in \mathbb{R}^{2}$, where $w=\left( v\cos z-\sin z\right)
\left( v\sin z+\cos z\right) ^{-1}$. Now, specializing in $v=\varepsilon
I_{2}$ with $\varepsilon =\pm 1$, we have 
\begin{equation*}
	w=\left( 
	\begin{array}{cc}
		\varepsilon \cosh t & \sinh t \\ 
		-\sinh t & \varepsilon \cosh t%
	\end{array}%
	\right) \left( 
	\begin{array}{cc}
		\cosh t & -\varepsilon \sinh t \\ 
		\varepsilon \sinh t & \cosh t%
	\end{array}%
	\right) ^{-1}=\left( 
	\begin{array}{cc}
		\varepsilon \text{sech}\left( 2t\right) & \tanh \left( 2t\right) \\ 
		-\tanh \left( 2t\right) & \varepsilon \text{sech}\left( 2t\right)%
	\end{array}%
	\right) \text{,}
\end{equation*}%
from which (\ref{arcsin}) follows.



\noindent\textsc{famaf} (Universidad Nacional de C\'ordoba) and \textsc{ciem} (Conicet) , Ciudad
	Universitaria, (X5000HUA) C\'{o}r\-doba, Argentina; eduardo.hulett@unc.edu.ar
	
	\smallskip
	
\noindent\textsc{fcefq}y\textsc{n} (Universidad Nacional de R\'{\i}o Cuarto) and Conicet, Argentina; paomoas@unc.edu.ar

\smallskip
	
\noindent\textsc{famaf}  (Universidad Nacional de C\'ordoba) and \textsc{ciem} (Conicet) , Ciudad
	Universitaria, (X5000HUA) C\'{o}r\-doba, Argentina; marcos.salvai@unc.edu.ar
  

\begin{thebibliography}{99}
	\bibitem{AgrachevLibroNuevo} A. Agrachev, D. Barilari, U. Boscain. A
	comprehensive introduction to sub-Riemannian geometry. From the Hamiltonian
	viewpoint. With an appendix by Igor Zelenko. Cambridge Studies in Advanced
	Mathematics, 181. Cambridge University Press, Cambridge, 2020.
	
	\bibitem{Alibro} A. Agrachev, Yu. Sachkov. Control theory from the geometric
	viewpoint. Encyclopaedia of Mathematical Sciences, 87. Control Theory and
	Optimization, II. Springer-Verlag, Berlin, 2004.
	
	\bibitem{Alek} D. Alekseevsky. Shortest and straightest geodesics in
	sub-Riemannian geometry. J. Geom. Phys. 155 (2020), 103713, 22 pp.
	
	\bibitem{AS} M. Anarella, M. Salvai. Infinitesimally helicoidal motions with
	fixed pitch of oriented geodesics of a space form. Acta Appl. Math. 179
	(2022), Paper No. 6, 19 pp.
	
	\bibitem{AM} C. Autenried, I. Markina. Sub-Riemannian geometry of Stiefel
	manifolds. SIAM J. Control Optim. 52 (2014) 939--959.
	
	\bibitem{BG} E. Berge, E. Grong. 
	On $\mathrm{G}_2$ and sub-Riemannian model spaces of step and rank three.
	Math. Z. 298 (2021) 1853-1885.
	
	\bibitem{BoscainCG} U. Boscain, T. Chambrion, J.-P. Gauthier. On the $K+P$
	problem for a three-level quantum system: optimality implies resonance. J.
	Dynam. Control Systems 8 (2002), no. 4, 547--572.
	
	\bibitem{Boscain} U. Boscain, F. Rossi. Invariant Carnot-Caratheodory
	metrics on $S^{3}$, $SO\left( 3\right) $, $SL\left( 2\right) $, and lens
	spaces. SIAM J. Control Optim. 47 (2008) 1851--1878.
	
	\bibitem{Brockett} R.W. Brockett. Explicitly solvable control problems with
	nonholonomic constraints, Proceedings of the 38th IEEE Conference on
	Decision and Control, vol. 1, 13-16, 1999.
	
	\bibitem{SK} M. Chemtov, S. Karigiannis. Observations about the Lie algebra $%
	\mathfrak{g}_{2}\subset s\mathfrak{o}\left( 7\right) $, associative
	3-planes, and $s\mathfrak{o}\left( 4\right) $ subalgebras. Expo. Math. 40
	(2022), 845--869.
	
	\bibitem{Domokos} A. Domokos, M. Krauel, V. Pigno, C. Shanbrom, M.
	VanValkenburgh. Length spectra of sub-Riemannian metrics on compact Lie
	groups. Pacific J. Math. 296 (2018), no. 2, 321--340.
	
	\bibitem{CDF} C. Draper Fontanals. Notes on $G_{2}$: the Lie algebra and the
	Lie group. Differential Geom. Appl. 57 (2018), 23--74.
	
	\bibitem{ESV} D. Emmanuele, M. Salvai, F. Vittone. M\"{o}bius fluid dynamics
	on the unitary groups. Regul. Chaotic Dyn. 27 (2022) 333--351.
	
	\bibitem{EngelBook} K.-J. Engel, R. Nagel. One-parameter semigroups for
	linear evolution equations. Graduate Texts in Mathematics, 194.
	Springer-Verlag, New York, 2000.
	
	\bibitem{Eschen} J.-H. Eschenburg. Geometry of octonions. Online notes,
	University of Augsburg, 2018.
	
	\bibitem{GMG} M. Godoy Molina, E. Grong. Riemannian and sub-Riemannian geodesic flows. 
	J. Geom. Anal. 27 (2017) 1260-1273.
	
	\bibitem{G} E. Grong. 	Model spaces in sub-Riemannian geometry.
	Commun. Anal. Geom. 29 (2021) 77-113.
	
	\bibitem{Harvey} F.R. Harvey. Spinors and calibrations. Perspectives in
	Mathematics, 9. Academic Press, Inc., Boston, MA, 1990.
	
	\bibitem{HMSl} K. H\"{u}per, I. Markina, F. Silva Leite. A Lagrangian
	approach to extremal curves on Stiefel manifolds. J. Geom. Mech. 13 (2021)
	55--72.
	
	\bibitem{JurBook} V. Jurdjevic. Geometric control theory. Cambridge Studies
	in Advanced Mathematics, 52. Cambridge University Press, Cambridge, 1997.
	
	\bibitem{JMF} V. Jurdjevic, I. Markina, F. Silva Leite. Extremal curves on
	Stiefel and Grassmann manifolds. J. Geom. Anal. 30 (2020) 3948--3978.
	
	\bibitem{RM} R. Montgomery. A tour of subriemannian geometries, their
	geodesics and applications. Mathematical Surveys and Monographs, 91.
	American Mathematical Society, Providence, RI, 2002.
	
	\bibitem{Neretin} Yu.A. Neretin, Pseudo-Riemannian symmetric spaces: Uniform
	realizations and open embeddings into Grassmannians. J. Math. Sci. (N. Y.)
	107 (2001), 4248--4264.
	
	\bibitem{O} G. Ovando. Lie algebras with ad-invariant metrics. A
	survey-guide. Rend. Semin. Mat., Univ. Politec. Torino 74 (2016) 243--268.
	
	\bibitem{AP} A. Podobryaev, Homogeneous geodesics in sub-Riemannian
	geometry. ESAIM: COCV 29 (2023) 11.
	
	\bibitem{Sachkov22} Yu. Sachkov. Left-invariant optimal control problems on
	Lie groups: classification and problems integrable by elementary functions.
	Russian Math. Surveys 77 (2022), no. 1, 99--163.
	
	\bibitem{S} D. Salamon, Th. Walpuski. Notes on the octonions. Proceedings of
	the Gokova Geometry-Topology Conference 2016, 1--85, Gokova
	Geometry/Topology Conference (GGT), Gokova, 2017.
	
	\bibitem{Toth} G.Z. T\'{o}th. On Lagrangian and Hamiltonian systems with
	homogeneous trajectories. J. Phys. A 43 (2010) 385206.
\end{thebibliography}
\end{document}